\documentclass[12pt,centertags,twoside]{amsart}
\usepackage{amsmath,amstext,amsthm,amscd,typearea,hyperref}
\usepackage{amssymb}
\usepackage{a4wide}
\usepackage{mathrsfs}
\usepackage{typearea}
\usepackage{charter}
\usepackage{pdfsync}
\usepackage[a4paper,width=16.5cm,top=3cm,bottom=3cm]{geometry}


\usepackage{xcolor}

\usepackage[utf8]{inputenc}

\def\1{{\bf 1}}

\def\End{{\rm End}}

\def\dbar{{\bar\partial}}
\def\R{{\mathbb R}}
\def\C{{\mathbb C}}
\def\Z{{\mathbb Z}}

\def\PM{{\mathcal{PM}}}
\def\Hom{{\rm Hom\, }}

\def\codim{{\rm codim\,}}

\def\Ok{{\mathcal O}}

\DeclareMathOperator{\supp}{supp}

\DeclareMathOperator{\im}{im}

\DeclareMathOperator{\res}{res}

\def\F{{\mathscr{F}}}
\def\G{{\mathscr{G}}}

\def\Dc{{\mathscr{D}}}

\def\elem{{e}}

\newcommand{\Res}{\text{\normalfont  Res}}
\newcommand{\sing}{\text{\normalfont  sing }}
\def\sf{{S}}

\def\be{\begin{equation}}
\def\ee{\end{equation}}

\newtheorem{thm}{Theorem}[section]
\newtheorem{lma}[thm]{Lemma}
\newtheorem{cor}[thm]{Corollary}
\newtheorem{prop}[thm]{Proposition}

\theoremstyle{definition}

\newtheorem{df}[thm]{Definition}

\theoremstyle{remark}

\newtheorem{preremark}[thm]{Remark}
\newtheorem{preex}[thm]{Example}

\newenvironment{remark}{\begin{preremark}}{\end{preremark}}
\newenvironment{ex}{\begin{preex}}{\end{preex}}

\numberwithin{equation}{section}

\allowdisplaybreaks
\emergencystretch=\maxdimen
\begin{document}

\title[Baum-Bott residue currents]{Baum-Bott residue currents}

\date{\today}

\author[Lucas Kaufmann \& Richard L\"ark\"ang \& Elizabeth Wulcan]
{Lucas Kaufmann \& Richard L\"ark\"ang \& Elizabeth Wulcan}

\address{L.\ Kaufmann \\ Institut Denis Poisson, CNRS, Universit\' e d'Orl\' eans, Rue de Chartres, B.P. 6759, 45067, Orl\' eans cedex 2, FRANCE}

\email{lucas.kaufmann@univ-orleans.fr}

\address{R.\ L\"ark\"ang, E.\ Wulcan \\ Department of Mathematics\\Chalmers University of Technology and the University of Gothenburg\\S-412 96
Gothenburg\\SWEDEN}

\email{larkang@chalmers.se, wulcan@chalmers.se}

\begin{abstract}
  Let $\F$ be a holomorphic foliation of rank $\kappa$ on a complex manifold $M$ of dimension $n$, let $Z$ be a compact connected component of the singular set of $\F$, and let $\Phi \in \C[z_1,\ldots,z_n]$ be a homogeneous symmetric polynomial of degree $\ell$ with $n-\kappa < \ell \leq n$.
  Given a locally free resolution of the normal sheaf of $\F$,
  equipped with Hermitian metrics and certain smooth connections, we construct an explicit current $R^\Phi_Z$ with support on $Z$ that represents the Baum-Bott residue $\res^\Phi(\F; Z)\in H_{2n-2\ell}(Z, \C)$ and is obtained as the limit of certain smooth representatives of $\res^\Phi(\F; Z)$.
  If the connections are $(1,0)$-connections and $\codim Z\geq \ell$, then $R^\Phi_Z$ is independent of the choice of metrics and connections.
  When $\F$ has rank one we give a more precise description of $R^\Phi_Z$ in terms of so-called residue currents of Bochner-Martinelli type. In particular, when the singularities are isolated, we recover the classical expression of Baum-Bott residues in terms of Grothendieck residues.
\end{abstract}

\maketitle

\section{Introduction}

Let $M$ be a complex manifold of dimension $n$, let $\F$ be a
holomorphic foliation of rank $\kappa$ on $M$ and denote by $N \F$ its
normal sheaf, see Section ~\ref{subsec:prelim-foliations} for the
definitions. Baum-Bott's vanishing theorem  asserts that, when $\F$ is
non-singular, all the characteristic classes of $N \F$ of degree
$\ell$ vanish when $n-\kappa < \ell \leq n$, see Theorem~
~\ref{thm:vanishing} below.

When $M$ is compact and $\F$ is singular, the vanishing theorem implies the following fundamental index theorem: for every  connected component $Z$ of the singular set of $\F$, $\sing \F$, and any homogeneous symmetric polynomial $\Phi \in \C[z_1,\ldots,z_n]$ of degree $\ell$ with $n-\kappa < \ell \leq n$, there exists a cohomology class $\Res^\Phi(\F;Z) \in H^{2\ell}(M,\C)$ depending only on the local behavior of $\F$ around $Z$ such that
\begin{equation*}
 \sum_{Z \subset \sing \F} \Res^\Phi(\F;Z) = \Phi (N \F) \quad \text{ in } \quad H^{2\ell}(M,\C),
\end{equation*}
where $\Phi (N \F)$ is the corresponding characteristic class of $N \F$, see \eqref{eq:chern-class-coherent-sheaf}. This should be seen as a localization formula for $\Phi (N \F)$ around the singularities of $\F$.

From the above formula, the question of computing the residues $\Res^\Phi(\F;Z)$ or finding
explicit representatives becomes natural. When $\F$ is of rank one and the
singular component $Z$ is a single point $p \in M$, the residue
$\Res^\Phi(\F;p)$ is actually a number which can be computed in terms
of the classical Grothendieck residue. More precisely, for any homogeneous symmetric polynomial $\Phi \in \C[z_1,\ldots,z_n]$ of degree $n$, if $z = (z_1,\ldots,z_n)$ is a local coordinate system centered at $p$ so that $\F$ is generated by a holomorphic vector field $X = \sum_{i=1}^n a_i(z) \frac{\partial}{\partial z_i}$ near $0$ with $\{a_1 = \cdots = a_n = 0\} = \{0\}$, then
   \begin{equation} \label{eq:grothendieck-residue}
 \Res^\Phi(\F;p) = \Res_0 \left[ \Phi\left(\left(\frac{\partial a_i}{\partial z_j}\right)_{ij}\right) \frac{dz_1 \wedge \ldots \wedge dz_n}{a_1\cdots a_n} \right],
 \end{equation}
where the right hand side denotes the usual Grothendieck residue, see
\cite[\S 8]{baum-bott}, \cite[Ch. 5]{griffiths-harris}, and Example
\ref{supper} and Remark ~\ref{rmk:grothendieck-residue} below.
That the class of $\Phi(N\F)$ may be represented as the sum of Grothendieck residues in this situation was proven
earlier in \cite{baum-bott-mero}, see also \cite{soares} for an elementary proof.

 For foliations of higher rank or with larger singular set, little is known. The available results are limited and rely on a reduction to the case of rank one foliations with isolated singularities, where the above formula can be used. For instance, in the particular case where $\Phi$ has degree $n-\kappa+1$, one can slice the foliation by suitable transverse sections on which the induced foliation has rank one and isolated singularities \cite{baum-bott,correa-lourenco}. Another approach is to use a continuity theorem together with a perturbation to a foliation with isolated singularities, see \cite{bracci-suwa}. Together with \eqref{eq:grothendieck-residue} this can be used to effectively compute residues of some rank-one foliations with large singular set.

\smallskip

The main goal of this work is to obtain explicit representatives of Baum-Bott residues in general, without any restriction on the rank of $\F$ nor on the dimension of its singular set. Our main result is that the class $\Res^\Phi(\F;Z)$ can be naturally represented by a certain current supported by $Z$. We call these currents  \textit{Baum-Bott (residue) currents}.

Assume that
\begin{equation} \label{eq:resolution-NF}
    0 \to E_N \stackrel{\varphi_N}{\longrightarrow} E_{N-1} \stackrel{\varphi_{N-1}}{\longrightarrow} \dots \stackrel{\varphi_2}{\longrightarrow} E_1 \stackrel{\varphi_1}{\longrightarrow} TM
    \stackrel{\varphi_0}{\longrightarrow} N\F \to 0,
\end{equation}
is a locally free resolution of $\Ok$-modules of $N\F$,
where $TM$ is the holomorphic tangent bundle on $M$ and $\varphi_0: TM
\to N \F$ is the canonical projection, and assume that  $TM, E_1, \ldots, E_N$ are equipped with
connections $D_0, D_1, \ldots, D_N$,
respectively; throughout, we tacitly assume that all connections are
smooth. Given a homogeneous symmetric polynomial $\Phi \in
\C[z_1,\ldots,z_n]$ of degree $\ell$ with $n-\kappa < \ell \leq n$,
consider the characteristic form
$r^\Phi(D):=(i/2\pi)^\ell\Phi(\Theta(D_N)|\ldots |\Theta(D_0))$ 
associated with the collection
$D=(D_0, \ldots, D_N)$, see Section \ref{texas} for the precise
definition.
Baum-Bott showed that if
$U$ is a neighborhood of a compact connected component $Z$ of
$\sing\F$ that deformation retracts to $Z$,  and $D$ is fitted to
\eqref{eq:resolution-NF} and $U$ in
a certain sense, see Section ~\ref{lucia}, then the restriction $r^\Phi_Z(D)$ of $r^\Phi(D)$ to
$U$
is a form of degree $2\ell$ with compact support in $U$ and thus it
defines an element in  $H_{2n-2\ell}(Z,\C)$; moreover the class of
$r^\Phi_Z(D)$ only depends on the local behaviour of $\F$ around $Z$. The \emph{homological Baum-Bott residue}  $\res^\Phi(\F;Z) \in H_{2n-2\ell}(Z,\C)$ is defined as the class of  $r_Z^\Phi(D)$.
If $M$ is compact, then by inclusion of $Z$ in $M$ and Poincar\'e
duality, $\res^\Phi(\F;Z)$ corresponds to a class in $H^{2\ell}(M,
\C)$; this is by definition $\Res^\Phi(\F;Z)$. For details, see Section ~\ref{lucia}.

\begin{thm} \label{thm:BB-current}
 Let $M$ be a complex manifold of dimension $n$,  let $\F$ be a holomorphic foliation of rank $\kappa$ on $M$, and let $\Phi\in \C[z_1,\ldots, z_n]$ be a homogeneous symmetric polynomial of degree $\ell$ with $n-\kappa<\ell\leq n$.
Assume that the normal sheaf $N\F$ of $\F$ admits a locally free
resolution of the form \eqref{eq:resolution-NF} on $M$,
and that $TM, E_1, \ldots, E_N$ are equipped with Hermitian metrics
and  $(1,0)$-connections $D^{TM}, D_1, \ldots, D_N$, respectively, and
assume that
$D^{TM}$ is torsion free.

Then for $\epsilon>0$
there are $(1,0)$-connections $\widehat{D}_0^\epsilon,
\widehat{D}_1^\epsilon, \dots,\widehat{D}_N^\epsilon$ on
$TM,E_1,\dots,E_0$, respectively, constructed from the Hermitian
metrics and
connections $D^{TM}, D_0, \ldots, D_N$,
such that
\begin{equation*}
    \lim_{\epsilon \to 0}
    \left(\frac{i}{2\pi}\right)^\ell\Phi(\Theta(\widehat{D}_N^\epsilon|\dots|\Theta(\widehat{D}_0^\epsilon))
  \end{equation*}
  exists as a current
  \[
    R^\Phi=\sum R^\Phi_Z,
    \]
where the sum runs over the connected components $Z$ of $\sing \F$.
For each $Z$, $R_Z^\Phi$ is a closed current of degree $2\ell$ with
support on $Z$
that only depends on \eqref{eq:resolution-NF} and the Hermitian
metrics and connections $D^{TM}, D_1, \ldots, D_N$
close to $Z$.
   If $\codim Z\geq \ell$, then $R_Z^\Phi$ is independent of the choice of metrics and connections.
Moreover, when $Z$ is compact,
$R^\Phi_Z$ represents the Baum-Bott residue $\res^\Phi(\F;Z) \in H_{2n-2\ell}(Z,\C)$.
\end{thm}

To construct the connections $\widehat D^\epsilon_k$ we first
construct connections $\widetilde D_k$ on $E_k|_{M\setminus \sing\F}$
and $D_{basic}$ on $N\F|_{M\setminus \sing\F}$ such that $D_{basic}$
is a so-called basic connection and $(\widetilde D_N, \ldots,
\widetilde D_0, D_{basic})$ is compatible with
\eqref{eq:resolution-NF} in a certain sense, see Definition
~\ref{def:basic-connection} and \eqref{eq:compatible}. The $\widetilde
D_k$ are defined only over $M\setminus \sing \F$, but their
singularity as we approach $\sing \F$ can be controlled. More precisely, they can be thought of as singular connections on $M$ with almost semi-meromorphic singularities  along $\sing\F$ in the sense of \cite{andersson-wulcan:asm}, cf.\ Lemma ~\ref{wedding}. The $\widehat D_k^\epsilon$ are then constructed as smoothings of the $\widetilde D_k$.

 It follows from this control of the singularities that the limits $R^\Phi$ of the characteristic forms $r^\Phi(\widehat D^\epsilon)$ exist and are so-called residue currents, or more precisely pseudomeromorphic currents in the sense of \cite{andersson-wulcan:crelle}, and can be seen as generalizations of the Grothendieck residue, see Section ~\ref{stream}. Note that we give meaning to the Baum-Bott current $R^\Phi_Z$ even when $Z$ is non-compact. Together with the dimension principle for pseudomeromorphic currents (Proposition \ref{prop:dimension-principle}),  the non-vanishing of certain  currents $R_Z^\Phi$ yields lower bounds on the dimension of the corresponding singular component $Z$, see Corollary \ref{cor:dim-Z}.

 The characteristic forms  $r^\Phi(\widehat D^\epsilon)$
 depend on the choices of Hermitian metrics and $(1,0)$-connections on
 the bundles in \eqref{eq:resolution-NF} and consequently so do the
 limits
 $R^\Phi$. In Section ~\ref{pomona} we give a description of this dependence. In particular, it follows that $R_Z^\Phi$ is independent of the metrics and connections if $\codim Z\geq \ell$.

 The existence of the currents $R^\Phi$ relies on the existence of a locally free resolution of $N\F$.
Such a resolution exists, e.g., if $M$ is a is a projective manifold. Also if $M$ is Stein, it exists after
replacing $M$ by some neighborhood of any compact subset, cf., e.g.,
\cite[Theorem 7.2.1]{hor:scv} or \cite[p.\ ~991]{andersson-wulcan:ens}.
In future work in progress, \cite{KLW-global}, we show that our construction may be combined with the global theory
of Chern classes of Green \cite{green} to obtain currents $R^\Phi$ without having to assume the existence of a global resolution.
We also obtain stronger results about independence of (local) locally free resolutions of $N\F$,
not just independence of the metrics and connections.

The construction of $R^\Phi$ is inspired by
\cite{LW:chern-currents-sheaves} where, given a locally free
resolution of a coherent analytic sheaf $\G$ whose support $\supp \G$
has positive codimension, explicit currents that represent the Chern
class of $\G$ with support on $\supp\G$ were defined as limits of
certain Chern forms.
Here we aim to represent characteristic classes of $N\F$, which has
full support, by currents supported by the proper analytic subset
$\sing \F$.  This is possible due to the existence of the special
connection $D_{basic}$ on $N\F|_{M\setminus\sing\F}$ that satisfies a certain vanishing theorem,
cf.\ Theorem \ref{thm:vanishing} and Lemma \ref{lemma:Dbasic} below.

When $\F$ has rank one we give a more precise description of the currents $R^\Phi$ in terms of so-called residue currents of Bochner-Martinelli type, see Theorem ~\ref{thm:BBBM}.  In particular, when $Z$ is a single point, we recover formula \eqref{eq:grothendieck-residue} above, see Corollary ~\ref{nordstan}.

\smallskip

The paper is organized as follows. In Section ~\ref{sec:prelim} we
introduce some notation and provide some necessary background on
complexes of vector bundles, characteristic classes, and holomorphic foliations. In Section ~\ref{bbres} we recall the construction of residues in \cite{baum-bott} and in Section ~\ref{stream} we gather some basic definitions and preliminary results on residue currents. In Section ~\ref{sec:basic-connection} we describe the construction of the connections $\widehat D^\epsilon_k$.  In Section ~\ref{sec:BB-currents} we show that the characteristic forms $r^\Phi_Z(\widehat D^\epsilon_k)$ have limits as pseudomeromorphic currents and prove a more precise and slightly more general formulation of Theorem ~\ref{thm:BB-current}, see Theorem ~\ref{train} and Corollary ~\ref{famine}; we also investigate the dependence of the Hermitian metrics and connections. Section ~\ref{sec:vector-field} is devoted to the rank one case.

\smallskip

\textbf{Acknowledgments.}
Part of this work was carried out while L.\ Kaufmann was at the Institute for Basic Science (under the grant IBS-R032-D1). He would like to thank IBS and Jun-Muk Hwang for providing excellent working conditions. He also benefited from visits to the Department of Mathematical Sciences of the University of Gothenburg and Chalmers University of Technology. He would like to thank the department for the support and hospitality.

R.\ L\"ark\"ang was partially supported by the Swedish Research Council (2017-04908).

E.\ Wulcan was partially supported by the Swedish Research Council (2017-03905) and the G\"oran Gustafsson Foundation for Research in Natural Sciences and Medicine.

The authors thank J. V. Pereira for helpful discussions.

\section{Holomorphic foliations, vector bundle complexes and characteristic classes}\label{sec:prelim}

Throughout the paper $M$ will be a complex manifold of dimension $n$.

\subsection{Holomorphic foliations} \label{subsec:prelim-foliations}

A \textit{holomorphic foliation}  $\F$ on $M$ is the data of a coherent analytic subsheaf $T \F$ of $TM$, called the \textit{tangent sheaf} of $\F$, such that

\begin{itemize}
\item[(i)]  $T \F$ is \emph{involutive}, that is, for any pair of local sections $u,v$ of $T \F$, the Lie bracket $[u,v]$ belongs to $T \F$;

\item[(ii)]  $\F$ is \emph{saturated}, that is, the \textit{normal sheaf} $N \F:=TM / T\F$ is torsion free.
\end{itemize}

The generic rank of $T \F$ is called the \textit{rank of $\F$}. Note
that $N\F$ is a coherent analytic sheaf. The \textit{singular set} of
$\F$ is, by definition, the smallest subset $\sing \F \subset M$
outside of which $N \F$ is locally free. It follows from our
definitions that $\sing \F$ is an analytic subset of $M$ of
codimension $\geq 2$. We say that $\F$ is \textit{regular} if $\sing
\F$ is empty.  By definition, the restriction of  $N\F$ to $M
\setminus \sing \F$ defines a regular foliation
whose normal sheaf is a holomorphic vector bundle of rank $n-\kappa$,
where $\kappa$ is the rank of $\F$. Moreover, by Frobenius Theorem, over $M \setminus \sing \F$, the bundle $T\F$ is locally given by vectors tangent to the fibers of a (local) holomorphic submersion.

\smallskip

The saturation property above is standard in the literature on
holomorphic foliations. It allows one to avoid ``artificial''
singularities and is convenient when studying the birational geometry
of foliations. This condition is equivalent to the fullness of $T\F$ required in \cite{baum-bott}. We note that we do not explicitly use this condition in most of our proofs, except in Section ~\ref{sec:vector-field}, so our main results can be applied to non-saturated foliations if necessary.

\subsection{Vector bundle complexes, connections, and superstructure}\label{superduper}
Consider a complex of smooth complex vector bundles
\begin{equation} \label{eq:augmentedEcomplex}
     0 \to E_N \xrightarrow[]{\varphi_N}  \cdots\xrightarrow[]{\varphi_1}  E_0\xrightarrow[]{\varphi_0} E_{-1} \to 0
   \end{equation}
   over $M$.
Following \cite{andersson-wulcan:ens}, we equip $E := \bigoplus_{k=-1}^N E_k$ with a superstructure by letting $E^+ =  \bigoplus_{2k} E_k$ (resp.\ $E^- =  \bigoplus_{2k + 1} E_k$) be the even (resp.\ odd) parts of $E = E^+ \oplus E^-$. This is a $\Z \slash 2 \Z$-grading that simplifies some of the formulas and computations. An endomorphism $\varphi \in \End(E)$ is even (resp.\ odd) if it preserves (resp.\ switches) the $\pm$-components.

The superstructure affects how form-valued endomorphisms act.
Assume that $\alpha=\omega\otimes \gamma$ is a form-valued section of $\End (E)$, where $\omega$ is a smooth form of degree $m$ and $\gamma$ is a section of
$\Hom(E_\ell,E_k)$.
We let $\deg_f \alpha = m$ and $\deg_e \alpha = k-\ell$ denote
the \emph{form} and \emph{endomorphism degrees}, respectively, of $\alpha$.
The \emph{total degree} is $\deg \alpha = \deg_f \alpha + \deg_e \alpha$.
If $\beta$ is a form-valued section of $E$, i.e., $\beta=\eta\otimes\xi$, where
$\eta$ is a smooth form, and $\xi$ is a section of $E$, both homogeneous in degree, then
the we define the action of $\alpha$ on $\beta$ by
\begin{equation}\label{thyme}
    \alpha(\beta) := (-1)^{(\deg_e \alpha)(\deg_f \beta)} \omega \wedge \eta \otimes \gamma(\xi).
\end{equation}
If furthermore, $\alpha'=\omega'\otimes \gamma'$, where $\gamma'$ is a holomorphic
section of $\End (E)$, and $\omega'$ is a smooth form, both homogeneous in degree, then
we define
\begin{equation}\label{rosemary}
    \alpha \alpha' := (-1)^{(\deg_e \alpha)(\deg_f \alpha')} \omega\wedge \omega' \otimes \gamma \circ \gamma'.
  \end{equation}

 The superstructure also affects how endomorphisms act on vector
 field-valued sections of $E$. If $\alpha$ is a
 section of $\End (E)$, $u$ is a vector field on $M$, and $\beta$ is a
 section of $E$, then
  \begin{equation}\label{estragon}
    \alpha(u\otimes \beta) = (-1)^{\deg\alpha} u \otimes \alpha(\beta).
  \end{equation}

  \smallskip

  Given a vector field $u$ on $M$, we let $i(u)$ denote contraction of a differential form on $M$ by $u$.
  Then $i(u)$ extends to form valued sections of $E$ and $\End (E)$ by letting
  \begin{equation*}
    i(u) (\eta \otimes \xi) = i(u) \eta \otimes \xi,
  \end{equation*}
  if $\eta$ is a smooth form and $\xi$ is a section of $E$ or $\End (E)$.
  It follows from \eqref{thyme} and \eqref{rosemary} that if $\alpha, \alpha'$ and $\beta$ are form-valued sections of $\End (E)$ and $E$, respectively, then
  \begin{equation}\label{parsley}
 i(u) \alpha(\beta) = i(u) \big (\alpha (\beta) \big ) - (-1)^{\deg \alpha} \alpha \big (i(u) \beta\big )
\end{equation}
and
 \begin{equation}\label{tennis}
 i(u) (\alpha\alpha') = i(u)\alpha \alpha'  + (-1)^{\deg \alpha} \alpha i(u) \alpha'.
\end{equation}

\smallskip

Assume that each $E_k$ is equipped with a connection $D_k$. Then there
is an induced connection $D_E$ on $E$, that in turn induces a connection $D_{\End}$ on $\End (E)$, defined by
\begin{equation}\label{meat}
D_\End\alpha = D_E\circ \alpha - (-1)^{\deg \alpha}\alpha \circ D_E.
  \end{equation}
This connection takes the superstructure into account and it satisfies the Leibniz' rule
\begin{equation}\label{wheat}
D_\End(\alpha\alpha') = D_\End \alpha \alpha' + (-1)^{\deg \alpha}\alpha D_\End \alpha'.
\end{equation}
Here $\alpha$ and $\alpha'$ are form-valued sections of $\End (E)$.
If $\alpha : E_k \to E_\ell$, we will sometimes write
\begin{equation} \label{eq:DEndDependence}
    D_{\End}\alpha = D_{D_k,D_\ell} \alpha
\end{equation}
when we need to specify the dependence on the connections.

Following \cite{baum-bott} we say that the collection of connections $(D_N, \ldots, D_{-1})$ is \emph{compatible} with \eqref{eq:augmentedEcomplex} if
\begin{equation} \label{eq:compatible}
    D_{k-1} \circ \varphi_k = -\varphi_k \circ D_k
  \end{equation}
  for $k=0, \ldots, N$. In terms of the induced connection $D=D_\End$ on $\End(E)$, the compatibility conditions simply become $D\varphi_k = 0$.  We note that \eqref{eq:compatible} differs by a sign from the original definition in \cite[Defintion~4.16]{baum-bott}.
  This is due to the superstructure convention, cf., \cite[Remark~4.2]{LW:chern-currents-sheaves}.

\smallskip

Assume that $\alpha$ is a scalar-valued section of $\End (E)$.
It will sometimes be convenient to consider $D_{\End}\alpha$ as a
section of
$\Hom(TM \otimes E,E)$.
Let $\Dc \alpha$ be the section of $\Hom(TM \otimes E,E)$ given by
\begin{equation} \label{jewel}
\Dc \alpha: u\otimes \beta \mapsto (-1)^{\deg \alpha} i(u) D_\End\alpha (\beta).
\end{equation}

In view of \eqref{estragon}, \eqref{tennis} and \eqref{wheat},  $\Dc$ satisfies the Leibniz rule
\begin{equation} \label{eq:nablaLeibniz}
\Dc(\alpha\alpha') = \Dc\alpha \alpha' + (-1)^{\deg \alpha}\alpha\Dc\alpha'.
\end{equation}

We extend $i(u)$ from $\End (E)$ to act on (form-valued) sections
$\Dc\alpha$ of $\Hom(TM\otimes E, E)$, by equipping $TM\otimes E$ with the same grading as $E$.
In particular,
 \begin{equation}\label{marit}
  \deg \Dc \alpha = \deg \alpha,
\end{equation}
and \eqref{parsley} and \eqref{tennis} hold also if $\alpha$ or $\alpha'$ is replaced by $\Dc \alpha$ or $\Dc \alpha'$.

If $\alpha$ and $\beta$ are sections of $\End(E)$ and $E$,
respectively, then $\Dc \alpha (\beta)$ is the section of $\Hom(TM, E)$ defined by
\begin{equation}\label{abba}
  \Dc \alpha (\beta): u \mapsto \Dc \alpha (u\otimes \beta).
\end{equation}

\subsection{Characteristic classes and forms}\label{texas}
Most of the material in this section can be found in
\cite[Sections~1 and ~4]{baum-bott}.

Let $E$ be a smooth complex vector bundle of rank $r$ over $M$ and let $\Phi \in \C[z_1,\ldots,z_n]$ be a homogeneous symmetric polynomial of degree $\ell \leq n$. Then there is a unique polynomial $\widehat \Phi$, such that
$$
\Phi(z_1,\ldots, z_n)=\widehat \Phi\big (\elem_1(z), \ldots, \elem_\ell(z)\big ),
$$
where $\elem_1,\ldots, \elem_\ell \in \C[z_1,\ldots,z_n]$ denote the elementary symmetric polynomials. The class $\Phi(E)\in H^{2\ell}(M, \C)$ is defined as
\begin{equation}\label{dragonfly}
  \Phi(E) = \widehat \Phi\big (c_1(E), \ldots, c_\ell(E)\big ),
\end{equation}
where $c_j(E)$ is the $j$th Chern class of $E$.

Assume that $E$ is equipped with a connection $D$. Then $\Phi(E)$ is the de Rham class of the closed $2\ell$-form
\[
  \left (\frac{i}{2\pi}\right )^\ell \Phi\big (\Theta (D) \big ),
\]
where $\Theta (D)$ is the curvature form of $D$ and we identify $\Phi$ with the corresponding invariant polynomial on (form-valued) $(r\times r)$-matrices. Note that with this identification,
\begin{equation}\label{butterfly}
\det \big (I + \Theta(D) \big ) = 1 + \elem_1\big (\Theta(D) \big ) + \cdots + \elem_n\big (\Theta(D) \big),
\end{equation}
and
\begin{equation}\label{schmetterling}
  \Phi\big (\Theta (D) \big ) =  \widehat \Phi\big (\elem_1(\Theta(D)), \ldots, \elem_\ell(\Theta(D))\big ).
\end{equation}
Note that $(i/2\pi)^j \elem_j (\Theta(D) ) $ is just the $j$th Chern form of $(E, D)$.

  \smallskip

Next, let $\G$ be a coherent analytic sheaf over $M$ and
assume that $\G$ admits a resolution by smooth complex vector bundles
\begin{equation}\label{snow}
  0 \to E_N \to \cdots \to E_0 \to \G \to 0.
  \end{equation}
  The \textit{total Chern class} of $\G$ is defined as the total Chern class of the virtual bundle $\sum_{k = 0}^N (-1)^k E_k$, i.e.,
  $$c(\G) = c \Big( \sum_{k = 0}^N (-1)^k E_k \Big) = \prod_{k=0}^N  c(E_k)^{(-1)^k} \in H^\bullet(M,\C);$$
  the class $c(\G)$ is independent of the chosen resolution, which follows from the construction of Chern classes of Green,
  \cite{green} (see also, e.g., \cite[\S 6]{baum-bott} in case the
  resolution is real analytic, and $M$ is
  compact).
We can write $c(\G)= 1 + c_1(\G) + \cdots + c_n(\G)$, where $c_j(\G) \in  H^{2j}(M,\C)$ is the $j$th \textit{Chern class} of $\G$.
If $\Phi \in \C[z_1,\ldots,z_n]$ is as above, then $\Phi(\G)\in H^{2\ell}(M, \C)$ is defined as
\begin{equation} \label{eq:chern-class-coherent-sheaf}
 \Phi(\G)= \widehat \Phi\big(c_1(\G),\ldots,c_\ell(\G)\big),
\end{equation}
cf.\ \eqref{dragonfly}.

Assume that the vector bundles in \eqref{snow} are equipped with connections
$D_0, \ldots, D_N$, and let $\Theta(D_k)$, $k=0 ,\ldots, N$, be the corresponding curvature forms.
Generalizing \eqref{butterfly} and \eqref{schmetterling} we let $\elem_j(\Theta(D_N) | \ldots  | \Theta(D_0))$ be the $2 j$-form defined by
\begin{equation} \label{eq:mixed-chern}
\prod_{k=0}^N \Big( \det \big[I + \Theta(D_k) \big] \Big)^{(-1)^k}  = 1 + \elem_1\big (\Theta(D_N) | \ldots  | \Theta(D_0)\big ) + \cdots + \elem_n\big (\Theta(D_N) | \ldots  | \Theta(D_0)\big),
\end{equation}
and we set
\begin{equation*}
  \Phi\big(\Theta(D_N) | \ldots  | \Theta(D_0)\big) = \widehat \Phi \big(\elem_1\big(\Theta(D_N) | \ldots  | \Theta(D_0)\big), \ldots, \elem_\ell\big(\Theta(D_N) | \ldots  | \Theta(D_0)\big) \big).
  \end{equation*}
  Then $\Phi(\Theta(D_N) | \ldots  | \Theta(D_0))$ is a closed $2\ell$-form and
  \[
 \left (\frac{i}{2\pi}\right )^\ell  \Phi\big(\Theta(D_N) | \ldots  | \Theta(D_0)\big)
\]
represents $\Phi(\G)$.
In particular, $(i/2\pi)^j\elem_j(\Theta(D_N) | \ldots  | \Theta(D_0))$ represents $c_j(\G)$.
If $\G$ is locally free (and \eqref{snow} is equipped with compatible connections) this is reflected on the level of forms in the following way:
\begin{lma}[\cite{baum-bott} - Lemma 4.22]  \label{lemma:exact-curvature}
  Assume that the complex \eqref{eq:augmentedEcomplex} is pointwise exact over some open set $U \subset M$. If $(D_N, \ldots, D_{-1})$ is a compatible collection of connections, then
 \begin{equation*}
\Phi\big(\Theta(D_N) | \ldots  | \Theta(D_0)\big) = \Phi\big(\Theta(D_{-1})\big) \quad \text{on} \quad U.
\end{equation*}
  \end{lma}

\section{Baum-Bott theory}\label{bbres}

The theory of Baum-Bott residues was developed in \cite{baum-bott},
extending the theory of rank one foliations in \cite{baum-bott-mero}
to general foliations.
Part of this theory may also be found in i.e., \cite{suwa:book}, where it is developed from
a slightly different perspective, making use of \v{C}ech-de Rham cohomology.

The main outcome of Baum-Bott's theory is the fact that
high degree characteristic classes $\Phi(N\F)$ of $N \F$ localize around $\sing \F$, cf.\ the introduction.
This is a consequence of a vanishing theorem for the normal bundle of a regular foliation due to the existence of special connections.
Recall that a connection is said to be of \emph{type $(1, 0)$}, or a \emph{$(1,0)$-connection} if its $(0, 1)$-part equals $\dbar$.

\begin{df} \label{def:basic-connection} (\cite {baum-bott} - Definition 3.24) Let $\F$ be a regular foliation on $M$ and let $\varphi_0: TM \to  N\F$ be the canonical projection. A connection $D$ on $N \F$ is \textit{basic} if it is of type $(1,0)$ and
\begin{equation} \label{eq:basicCondition}
i(u)D( \varphi_0 v) =  \varphi_0 [u,v]
\end{equation}
for any smooth sections $u$ of $T \F$ and $v$ of $TM$.
\end{df}

It is not hard to see that basic connections always exist, see
\cite[\S 3]{baum-bott} and also Lemma~\ref{lemma:Dbasic} below.

\begin{thm}[Baum-Bott's Vanishing theorem,
  \cite{baum-bott} - Proposition ~3.27] \label{thm:vanishing}
  Let $\F$ be a regular foliation of rank $\kappa$ on a complex
  manifold $M$ of dimension $n$.
  If $D$ is a basic connection on $N \F$ and $\Theta(D)$ denotes its
  curvature form,  then
 $$\Phi\big(\Theta(D)\big) = 0 \quad \text{on} \quad M$$
 for every homogeneous symmetric polynomial $\Phi \in \C[z_1,\ldots,z_n]$ of degree $\ell$ with $n-\kappa < \ell \leq n$.
\end{thm}

\subsection{Baum-Bott residues}\label{lucia}

In the presence of singularities, one cannot work directly with connections on $N\F$, so the use of suitable resolutions is necessary.

Let $Z$ be a compact connected component of $\sing \F$. Then one can
find an open neighborhood $U$ of $Z$ in $M$ such that $U\cap \sing
\F=Z$ and $Z$ is a deformation retract of $U$, and a locally free resolution
 \begin{equation}\label{shoe}
    0 \to E_N \xrightarrow[]{\varphi_N}  \cdots\xrightarrow[]{\varphi_1}  E_0 \xrightarrow[]{\varphi_0} N \F
    \to 0
   \end{equation}
of $\mathcal A$-modules on $U$, where $\mathcal A$ denotes the sheaf of germs of (complex-valued) real analytic functions, cf.\ \cite[Proposition~6.3]{baum-bott}.
The data $\beta=(U,(E_N,\ldots,E_0), (\varphi_N, \ldots, \varphi_0))$ is a called a \emph{$Z$-sequence}.\footnote{In fact, in \cite{baum-bott} the notion $Z$-sequence has a slightly different meaning. There, the bundles in \eqref{shoe} are assumed
to be smooth vector bundles on $U$, and the morphisms in the complex
are only assumed to exist and be pointwise exact in $U\setminus Z$. It
is clear that a locally free resolution of $N\F$ of $\mathcal
A$-modules gives rise to a $Z$-sequence in this sense, cf.\
\cite[Section~7]{baum-bott}.}

Now assume that the vector bundles $N\F|_{U\setminus Z}, E_0, \ldots, E_N$ are equipped with connections $D_{-1}, D_0, \ldots, D_N$. Following \cite{baum-bott} we say that the collection $(D_N, \ldots, D_{-1})$ is \emph{fitted} to $\beta$
if $D_{-1}$ is a basic connection and $(D_N, \ldots, D_{-1})$ is
compatible with \eqref{shoe} over $U \setminus \Sigma$ for some
compact neighborhood $\Sigma$ of $Z$, i.e., \eqref{eq:compatible}
holds over $U \setminus \Sigma$ .

From Theorem~ \ref{thm:vanishing} and Lemma ~\ref{lemma:exact-curvature} it follows that if  $(D_N, \ldots, D_{-1})$ is fitted to $\beta$ and $\Phi\in\C[z_1,\ldots, z_n]$ is a homogeneous symmetric polynomial of degree $\ell$ with $n-\kappa<\ell\leq n$, then
$$\Phi\big(\Theta(D_N) | \ldots  | \Theta(D_0)\big)=\Phi\big (\Theta (D_{-1})\big )$$
vanishes in $U\setminus \Sigma$, where $\Sigma$ is as above.
In particular, this is a closed compactly supported differential form on $U$.
Since $Z$ is a deformation retract of $U$, the homology groups of $U$ and $Z$ are naturally isomorphic. Composing this isomorphism with the Poincaré duality $H^{2\ell}_c(U,\C) \simeq H_{2n-2\ell}(U,\C)$ yields an isomorphism $H^{2\ell}_c(U,\C) \simeq H_{2n-2\ell}(Z,\C)$.
Now $\res^\Phi(\F;Z) \in H_{2n-2\ell}(Z,\C)$ is defined as the class of
  \begin{equation}\label{snowboard}
    \left ( \frac{i}{2\pi}\right )^n
   \Phi\big (\Theta(D_N)|\ldots|\Theta(D_0)\big )
 \end{equation}
 in $H^{2 \ell}_c(U,\C)$
 under this isomorphism.
It is proved in \cite[Sections~5,6,7]{baum-bott} that the class of \eqref{snowboard} is independent of the choice of $Z$-sequence and fitted connections, and that it only depends on the local behaviour of $\F$ around $Z$.

When $M$ is compact, the compactly supported $2 \ell$-form \eqref{snowboard}
extends naturally to a closed form on $M$. It follows from the
definition that the corresponding de Rham class
$\Res^\Phi(\F;Z) \in H^{2\ell}(M,\C)$, is the image of $\res^\Phi(\F;Z)$ under the composition of the map $\iota_* : H_{2n-2\ell}(Z,\C) \to H_{2n-2\ell}(M,\C)$ induced by the inclusion $\iota: Z \hookrightarrow M$ and the Poincaré duality $ H_{2n-2\ell}(M,\C) \simeq  H^{2\ell}(M,\C)$.

\section{Residue currents}\label{stream}

We say that a function $\chi:\R_{\geq 0}\to \R_{\geq 0}$ is a \emph{smooth approximant
  of the characteristic function} $\chi_{[1,\infty)}$ of the interval
  $[1,\infty)$ and write
  \[
    \chi \sim \chi_{[1,\infty)}
  \]
 if $\chi$ is smooth, increasing and $\chi(t) \equiv 0$ for $t \ll 1$ and
 $\chi(t) \equiv 1$ for $t \gg 1$.
\begin{remark}\label{elsa}
Note that if $\chi \sim \chi_{[1,\infty)}$ and $\widehat \chi = \chi^j$,
then $\widehat \chi \sim \chi_{[1,\infty)}$ and
\begin{equation*}
d\widehat\chi =  j\chi^{j-1} d\chi ~ \text{ and } \dbar\widehat\chi =  j\chi^{j-1} \dbar\chi
\end{equation*}
\end{remark}

\subsection{Pseudomeromorphic currents}\label{hemma}
Let $f$ be a (generically nonvanishing) holomorphic function on a
(connected)
complex manifold $M$.
Herrera and Lieberman, \cite{HL}, proved that the \emph{principal value}
\begin{equation*}
\lim_{\epsilon\to 0}\int_{|f|^2>\epsilon}\frac{\xi}{f}
\end{equation*}
exists for test forms $\xi$ and defines a current, that we with a slight abuse of notation denote by $1/f$.
It follows that $\dbar(1/f)$ is a current with support on the zero set
$Z(f)$ of $f$; such a current is called a \emph{residue current}.
Assume that $\chi\sim\chi_{[1,\infty)}$ and that
 $s$ is a generically
nonvanishing holomorphic section of a Hermitian vector bundle such that
$Z(f)\subseteq \{s = 0\}$.
Then
\begin{equation*}
  \frac{1}{f}=\lim_{\epsilon\to 0} \frac{\chi(|s|^2/\epsilon)}f ~~~~~ \text{
    and } ~~~~~
  \dbar\left (\frac{1}{f}\right )=\lim_{\epsilon\to 0} \frac{\dbar\chi(|s|^2/\epsilon)}f,
  \end{equation*}
  see, e.g., \cite{andersson-wulcan:asm}.
  In particular, the limits are independent of $\chi$ and $s$. Note
  that $\chi(|s|^2/\epsilon)$ vanishes identically in a neighborhood
  of $\{s=0\}$, so that $\chi(|s|^2/\epsilon)/f$ and $\dbar\chi(|s|^2/\epsilon)/f$ are smooth.
More generally, if $f$ is a generically non-vanishing holomorphic section of a line bundle $L\to M$ and $\omega$ is an $L$-valued smooth form, then the current $\omega \slash f$ is well-defined. Such currents are called \textit{semi-meromorphic}, cf.\ \cite[Section~4]{andersson-wulcan:asm}.

In the literature there are various generalizations of principal value currents and residue
currents.
In particular, Coleff and Herrera \cite{CH} introduced
products like
\begin{equation}\label{apan2}
  \frac{1}{f_m}\cdots \frac{1}{f_{r+1}} \dbar \frac{1}{f_{r}} \wedge \cdots \wedge \dbar \frac{1}{f_{1}}.
\end{equation}
If $\codim Z_f=m$, where $Z_f=\{f_1=\cdots = f_m=0\}$, then the \emph{Coleff-Herrera product} $\dbar(1/f_m)\wedge\cdots\wedge \dbar (1/f_1)$ is anti-commutative in the  factors and has support on $Z_f$.

One application of residue currents in general has been to provide explicit or canonical representatives
of cohomology classes. In particular, the Coleff-Herrera product has been used to
provide explicit canonical representatives in so-called moderate cohomology, \cite{DS}.
\begin{ex}\label{supper}
  Assume that $f_1,\ldots, f_n$ are holomorphic functions in some neighborhood $U$ of $p\in M$, such that $Z_f=\{p\}$. Let $\eta$ be a holomorphic $(n,0)$-form on $U$.
  Then the action of
  $1/(2\pi i)^n\dbar (1/f_n)\wedge \cdots \wedge \dbar (1/f_1)$ on $\eta$ is given by the integral
    \[
   \frac{1}{(2\pi i)^n}   \int_{\Gamma_\epsilon} \frac{\eta}{f_1\cdots f_n},
      \]
      where $\Gamma_\epsilon:=\{|f_1|=\epsilon, \ldots, |f_n|=\epsilon\}$ is oriented
      by $d(\arg f_n)\wedge \cdots \wedge d(\arg f_1)>0$, for a sufficiently small
      $\epsilon >0$.
This equals, by definition,
the Grothendieck residue
\[
\Res_p \left [\frac{\eta}{f_1\cdots f_n}\right],
\]
see \cite[Ch. 5]{griffiths-harris}.
        \end{ex}

        In \cite{andersson-wulcan:crelle} the sheaf  $\PM_M$ of {\it
          pseudomeromorphic currents} on $M$ was introduced in order to obtain a coherent approach to questions about residue and
principal value currents; it consists of direct images under holomorphic mappings of products of test forms
and currents like \eqref{apan2}.
See, e.g., \cite[Section~2.1]{andersson-wulcan:asm} for a precise
definition.
The sheaf $\PM_M$ is closed under $\partial$ and $\dbar$ and under multiplication by
smooth forms.
Pseudomeromorphic currents have a geometric
nature, similar to closed positive (or normal) currents.  For instance, they satisfy the following dimension principle,  see \cite[Corollary 2.4]{andersson-wulcan:crelle}.

\begin{prop} \label{prop:dimension-principle}
Let $\mu$ be a pseudomeromorphic current of bidegree $(p,q)$ on $M$.  If $\mu$ is supported on a subvariety of $M$ codimension strictly larger than
$q$,  then $\mu = 0$.
\end{prop}

The sheaf $\PM_M$ admits natural restrictions to constructible subsets
of $M$.
In particular, if $W$ is a subvariety of the open subset $U\subseteq M$,
and $s$ is a holomorphic section of a Hermitian vector bundle such that $\{s = 0\} = W$, then the restriction to
$U\setminus W$ of a pseudomeromorphic current $\mu$ on $U$ is
the pseudomeromorphic current on $U$ defined by
\[
    \1_{U\setminus W} \mu := \lim_{\epsilon \to 0} \chi(|s|^2/\epsilon) \mu|_{U},
\]
where $\chi \sim \chi_{[1,\infty)}$.
It follows that
    \[\1_{W} \mu := \mu -
    \1_{U\setminus W} \mu\]
  has support on $W$.
  These definitions are independent of the choice of $s$ and $\chi$.

\subsection{Almost semi-meromorphic currents}\label{asmsec}
We refer to \cite[Section~4]{andersson-wulcan:asm} for details of the
results mentioned in this section.

We say that a  current $a$ is \textit{almost semi-meromorphic} in $M$, $a\in ASM(M)$,  if there exists a modification $\pi: M' \to M$ and a semi-meromorphic current $\omega \slash f$  on $M'$ such that $a = \pi_* \big( \omega \slash f \big)$. More generally, if $E$ is a vector bundle over $M$, an $E$-valued current $a$ is \textit{almost semi-meromorphic} on $M$ if $ a = \pi_* \big( \omega \slash f \big)$, where $\pi$ is as above, $\omega$ is a smooth form with values in $L \otimes \pi^*E$ and $f$ is a holomorphic section of a line bundle $L\to M'$.

Clearly almost semi-meromorphic currents are pseudomeromorphic. In particular, if $a\in ASM(M)$, then $\partial a$ and $\dbar a$ are pseudomeromorphic currents on $M$.

\begin{lma}[Proposition~4.16 in \cite{andersson-wulcan:asm}]\label{gold}
  Assume that $a\in ASM (M)$ is smooth in $M\setminus W$, where $W$ is
  subvariety of $M$. Then $\partial a\in ASM (M)$ and $\1_{M\setminus W}\dbar a\in ASM(M)$.
\end{lma}

Given $a\in ASM(M)$, let
$ZSS(a)$ (the \emph{Zariski-singular support}) denote the smallest Zariski-closed set $V\subset M$ such that $a$ is smooth outside $V$. The pseudomeromorphic current $r(a):=\1_{ZSS(a)} \dbar a$ is called the \emph{residue} of $a$.

Almost semi-meromorphic currents have the so-called \emph{standard
  extension property (SEP)} meaning that $\1_Wa=0$ in $U$ for each
subvariety $W\subset U$ of positive codimension, where $U$ is any open
set in $M$.
In particular, if $a\in ASM(M)$, $\chi \sim \chi_{[1,\infty)}$, and $s$ is any generically non-vanishing holomorphic section of a Hermitian
vector bundle over $M$,
then
\begin{equation*}
    \lim_{\epsilon\to 0} \chi(|s|^2/\epsilon) a= a.
\end{equation*}
It follows in view of Lemma ~\ref{gold} that, if $\{s=0\}\supset ZSS(a)$, then
\begin{equation*}
 r(a) = \lim_{\epsilon\to 0}  \dbar \chi(|s|^2/\epsilon) \wedge a = \lim_{\epsilon\to 0}  d \chi(|s|^2/\epsilon) \wedge a.
\end{equation*}

\begin{remark}\label{bronze}
  Note that it follows from above
  that if $a\in ASM(M)$ is smooth outside the subvariety $W\subset M$, $\chi \sim \chi_{[1,\infty)}$, and $s$ is a generically nonvanishing holomorphic section of a Hermitian vector bundle over $M$ such that $\{s=0\}\supset W$,
  then the smooth forms
  \[
\chi(|s|^2/\epsilon) a, ~~~~~ \partial \chi(|s|^2/\epsilon) \wedge a, ~~~~~ \dbar \chi(|s|^2/\epsilon) \wedge a 
    \]
    have limits as pseudomeromorphic currents and the limits are independent of the choice of $\chi$ and $s$.
 \end{remark}

If $a_1, a_2 \in ASM(M)$, then $a_1+a_2\in ASM(M)$, and moreover there
is a well-defined product $a_1\wedge a_2\in ASM(M)$, so that $ASM (M)$
is an algebra over smooth forms, see \cite[Section~4.1]{andersson-wulcan:asm}. Note that
if $\chi \sim \chi_{[1,\infty)}$ and $s$ is a generically nonvanishing
holomorphic section of a Hermitian vector bundle such that $\{s=0\}$
contains the Zariski-singular supports of $a_1$ and $a_2$,  then
$a_1\wedge a_2$ is the limit of the smooth forms $\chi(|s|^2/\epsilon) a_1\wedge a_2.$

  \subsection{Residue currents of Bochner-Martinelli type}\label{bmx}
  Let us describe the construction of residue currents in \cite{A},
  see also \cite[Example~4.18]{andersson-wulcan:asm}. Let $f$ be a
  holomorphic section of the dual bundle $E^*$ of a Hermitian vector bundle $E\to M$ and let $\sigma$ be the minimal inverse of $f$, i.e., the section of $E$ over $M\setminus Z_f$ of minimal norm such that $f \sigma=1$; here $Z_f$ denotes the zero set of $f$. Moreover consider the section
  $$
  u^f:= \sum_{\ell\geq 0} \sigma(\dbar\sigma)^\ell
  $$
  of $\Lambda (E\oplus T^*_{0,1}(M))$;
  note that $\dbar \sigma$ has even degree in $\Lambda (E\oplus
  T^*_{0,1}(M))$, cf.\ \cite[Section~1]{A}.
One can show that $\sigma$ has an extension as an almost
semi-meromorphic current on $M$, see, e.g., the proof of Lemma~2.1 in \cite{LW:chern-currents-sheaves}.
Thus, if $\chi\sim \chi_{[1,\infty)}$ and $s$ is a generically non-vanishing holomorphic section of a Hermitian vector bundle over $M$ such that $\{s=0\}\supset Z_f$, then
$$
  R^f := r(u^f)=\lim_{\epsilon \to 0} \dbar\chi (|s|^2/\epsilon)\wedge u^f
  $$
is a pseudomeromorphic current on $M$ with support on $Z_f$. We let
  $R^f_k$ denote the component of
  $R^f$ that takes values in $\Lambda^kE$. Then $R^f_k$ has bidegree $(0,k)$.
This current first appeared in \cite{A}. If $E$ is trivial and
equipped with the trivial metric, then the coefficients are residue currents of Bochner-Martinelli type in the sense of \cite{passare-tsikh-yger}.

\begin{ex}\label{complete}
  Assume that
    $f=f_1e_1^*+\cdots + f_m e_m^*$, where $e_1^*, \ldots, e_m^*$ is the dual frame of a local frame $e_1, \ldots, e_m$ for $E$. Moreover assume that the codimension of $Z_f$ is $m$ (so that $f$ defines a complete intersection), then
  $$
  R^f=R^f_m= \dbar \frac{1}{f_m}\wedge \cdots \wedge \dbar \frac{1}{f_1} \wedge e_1\wedge \cdots \wedge e_m,
  $$
  see \cite[Theorem~1.7]{A} and \cite [Theorem~4.1]{passare-tsikh-yger}.
    \end{ex}

\section{Construction of connections}\label{sec:basic-connection}

Assume that $\F$ is a holomorphic foliation of rank $\kappa$ on $M$ and that
\eqref{eq:resolution-NF}
is a locally free resolution of $\Ok$-modules of the normal sheaf $N\F$ of $\F$ on $M$, i.e., the vector bundle complex is pointwise exact outside $\sing\F$ and the associated sheaf complex of holomorphic sections is exact. Here $\varphi_0: TM \to N \F$ is the canonical projection. From the exactness, it follows that $\im \varphi_1 = \ker \varphi_0 = T \F.$
Recall that $N\F$ is a holomorphic vector bundle over $M \setminus \sf$, where we use the shorthand notation $\sf=\sing\F$.
Throughout will use the
superstructure and sign conventions described in Section ~\ref{superduper}, with the convention $E_0=TM$ and $E_{-1}|_{M\setminus S}=N\F|_{M\setminus S}$.

In this section we construct a collection of  connections
that will be essential in the construction of Baum-Bott currents in Section ~\ref{sec:BB-currents}.
To this end, we assume that $TM, E_1, \ldots, E_N$ are equipped with Hermitian metrics and connections $D^{TM}, D_1,\ldots, D_N$,\footnote{Note that the connection $D_k$ is not necessarily the Chern connection of the metric on $E_k$.} respectively, such that $D^{TM}$ is a $(1,0)$-connection. Moreover, assume that $D^{TM}$ is
\textit{torsion-free},\footnote{Such connections $D^{TM}$ with the
  desired properties always exist. This can be easily seen to hold locally and a simple argument with a partition of unity yields a connection with the desired properties over the whole manifold $M$. If $M$ is K\"ahler, one can take $D^{TM}$ to be the Chern connection on $TM$, which is torsion-free.} that is
\begin{equation} \label{eq:torsion-free-def}
i(u) D^{TM} v - i(v) D^{TM} u = [u,v]
\end{equation}
for any pair of vector fields $u$ and $v$ on $M$.

Starting from these we will construct a basic connection $D_{basic}$ on $N\F|_{M\setminus \sf}$ and connections $\widetilde D_k$ on $E_k|_{M\setminus \sf}$ so that
$(\widetilde D_N, \ldots, \widetilde D_0, D_{basic})$ is compatible with \eqref{eq:resolution-NF} over $M\setminus \sf$. Next, by a choice of $\chi\sim\chi_{[1,\infty)}$ and a generically nonvanishing holomorphic section $s$ of a Hermitian vector bundle, we will construct connections $\widehat D^\epsilon_k$ on $M$ that coincide with $\widetilde D_k$ outside a neighborhood of $\{s=0\}$.
In particular, if we replace $M$ by a neighborhood of a compact connected component $Z$ such that $(M,(E_N, \ldots, E_1, TM), (\varphi_N, \ldots, \varphi_0))$
is a $Z$-sequence, then we can choose $s$ so that
$(\widehat D_N^\epsilon, \ldots, \widehat D_0^\epsilon, D_{basic})$ is fitted to it
for $\epsilon$ small enough.

\subsection{The connections $D_{basic}$ and $\widetilde D_k$
   on $M\setminus \sf$}\label{pomona}

For $k=1,\ldots,N$, we let $\sigma_k: E_{k-1}\to E_k$
be the \emph{minimal inverse}  of $\varphi_k$. These are smooth vector bundle morphisms defined outside the analytic set $Z_k\subset \sf$ where $\varphi_k$ does not have optimal rank and are determined by the following properties:
\begin{equation*}
\varphi_k \sigma_k \varphi_k = \varphi_k, \quad \im \sigma_k \perp \im \varphi_{k+1} \quad \text{and} \quad  \sigma_{k+1} \sigma_k = 0.
\end{equation*}
Note that $\varphi_k$ and
$\sigma_k$ have odd degree with respect to the superstructure.
It follows from the definition of $\sigma_k$ that in $M\setminus \sf$
\begin{equation} \label{eq:idSigma}
    I_{E_k} = \varphi_{k+1}\sigma_{k+1} + \sigma_k \varphi_k,
\end{equation}
for $1\leq k \leq N$, with the convention $\varphi_{N+1}=0$ and $\sigma_{N+1}=0$, and
\begin{equation}\label{ortho}
    \pi_0 := I-\varphi_1\sigma_1 : TM \to TM
\end{equation}
is the orthogonal projection onto $(\im \varphi_1)^\perp = (T
\F)^\perp\subset TM$.

\smallskip

 We start by modifying the connection $D^{TM}$ on $TM$ into a connection which will ultimately induce the desired  basic connection on $N \F|_{M \setminus \sf}$.
The vector bundle $TM$ carries a canonical one-form valued section, that we denote by $dz \cdot \partial/\partial z$, which is induced by the identity morphism on $TM$, viewed
as an element of $T^*M \otimes TM \cong \Hom(TM,TM)$. It is defined as
\begin{equation}\label{heart}
  dz \cdot \frac{\partial}{\partial z} := \sum_{k=1}^n dz_k \, \frac{\partial}{\partial z_k},
\end{equation}
where $(z_1,\dots,z_n)$ are local holomorphic coordinates on $M$. It is easy to see that the definition is independent of the choice of coordinates. It readily follows that, for a vector field $u$,  one has
\begin{equation} \label{eq:identity-TM}
i(u)(dz \cdot \partial/\partial z) = u.
\end{equation}

Now on $TM|_{M\setminus \sf}$ let
\begin{equation} \label{eq:Dvarphi1}
    D_0 = D^{TM} + \Dc \varphi_1\, \sigma_1 (dz \cdot \partial/\partial z),
  \end{equation}
  where $\Dc$ is as in Section ~\ref{superduper} and $D_{\End}$ is induced by $D^{TM}$ and $D_1$.
  Since $dz \cdot \partial/\partial z$ has bidegree $(1,0)$, it follows that
  \begin{equation}\label{boat}
    b:=\Dc \varphi_1 \sigma_1 (dz \cdot \partial/\partial z)
    \end{equation}
    is a smooth $(1,0)$-form in $M\setminus \sf$ with values in
$\End(TM|_{M\setminus \sf})$.
    Thus, since $D^{TM}$ is a $(1,0)$-connection,  $D_0$ is a well-defined $(1,0)$-connection on $TM|_{M\setminus \sf}$.

\begin{lma} \label{lemma:Dvarphi1}
    If $u$ is a smooth section of $T\F |_{M\setminus \sf}$ and $v$ is a smooth section of $TM|_{M\setminus \sf}$,
    then
    \begin{equation}\label{blue}
        i(u) D_0 v = [u,v] \mod \im \varphi_1.
    \end{equation}
\end{lma}

\begin{proof}
Consider the section $ i(u) ( \Dc \varphi_1 \sigma_1 (dz \cdot
\partial/\partial z))$ of $\End (TM|_{M\setminus \sf})$. By
\eqref{marit}, \eqref{parsley}, \eqref{tennis}, and
\eqref{eq:identity-TM},
\begin{equation}\label{afternoon}
  i(u) ( \Dc \varphi_1 \sigma_1 (dz \cdot \partial/\partial z))= \Dc \varphi_1 \sigma_1 u.
\end{equation}
Since $\im \varphi_1 = T \F$, we can locally on $M\setminus S$ write $u = \varphi_1 \beta$ for some section $\beta$ of $E_1$, and thus by \eqref{eq:idSigma},
    \begin{equation*}
      \Dc \varphi_1 \sigma_1 u = \Dc \varphi_1 \sigma_1 \varphi_1 \beta  = \Dc \varphi_1  \beta - \Dc \varphi_1  \varphi_2 \sigma_2 \beta.
    \end{equation*}
    Since $\varphi_1  \varphi_2 = 0$, it follows from \eqref{eq:nablaLeibniz} that
    $$
    \Dc \varphi_1 \varphi_2 = \varphi_1 \Dc \varphi_2,
    $$
    and thus
 \begin{equation}\label{pink}
        i(u)\big (\Dc \varphi_1 \sigma_1 (dz \cdot \partial/\partial z) \big )= \Dc \varphi_1 \beta\mod \im \varphi_1.
      \end{equation}

      Now apply this to $v$. By \eqref{abba}, \eqref{jewel}, \eqref{meat}, \eqref{parsley},
      \begin{multline}\label{yellow}
        (\Dc \varphi_1 \beta) v =
        \Dc \varphi_1 (v\otimes\beta) = - i(v)D_\End \varphi_1 (\beta)=
        -i(v) (D_\End \varphi_1 \beta) = \\
        - i(v)\big (D^{TM}(\varphi_1 \beta)\big) - i(v) (\varphi_1 D_1\beta) =
        -i(v) D^{TM} u \mod \im \varphi_1.
      \end{multline}
     By combining \eqref{pink} and \eqref{yellow} we get
      \begin{equation*}
        i(u) D_0 v = i(u) D^{TM} v + i(u)\big ( (\Dc \varphi_1 \sigma_1 (dz \cdot \partial/\partial z) \big ) v = i(u) D^{TM} v - i(v) D^{TM} u  \mod \im\varphi_1.
    \end{equation*}
   Now \eqref{blue}  follows from the torsion-freeness of $D^{TM}$, cf.\  \eqref{eq:torsion-free-def}.
 \end{proof}

 \smallskip

 Next we will use the connection $D_0$ to define a basic connection on
 $N\F|_{M \setminus \sf}$. Recall that $\varphi_0$ is surjective over
 $M\setminus S$.
 For a section $\varphi_0 v$ of  $N\F|_{M \setminus \sf}$ we let
   \begin{equation} \label{eq:DbasicDef}
        D_{basic}(\varphi_0v) := - \varphi_0 D_0( \pi_0 v),
      \end{equation}
      where $\pi_0$ is as in \eqref{ortho}.
    This is a well-defined $(1,0)$-connection on $N\F|_{M \setminus \sf}$ since $D_0$ is a $(1,0)$-connection on $TM|_{M\setminus \sf}$ and $\varphi_0$ and $\pi_0$ have the same kernel, namely $\im \varphi_1$. The minus sign in \eqref{eq:DbasicDef} is necessary for $D_0$ to define a connection, in view of the superstructure, cf.\ \eqref{meat}.

\begin{lma} \label{lemma:Dbasic}
    The connection $D_{basic}$ is a basic connection on $N\F|_{M \setminus \sf}$.
\end{lma}

\begin{proof}
  We saw above that $D_{basic}$ is a $(1,0)$-connection.
  It remains to prove that \eqref{eq:basicCondition} holds for any smooth sections $u$ and $v$ of $T\F|_{M\setminus \sf}$ and $TM|_{M\setminus \sf}$, respectively.
  Due to the superstructure, cf.\ \eqref{parsley} and \eqref{tennis},
  \begin{equation}\label{elephant}
    i(u) D_{basic}( \varphi_0  v) = \varphi_0 i(u)  D_0( \pi_0 v).
    \end{equation}
    By Lemma ~\ref{lemma:Dvarphi1} the right hand side of \eqref{elephant} equals
    \begin{equation*}
      \varphi_0 [u,\pi_0 v] = \varphi_0 [u,v] - \varphi_0[u,  \varphi_1 \sigma_1 v]=  \varphi_0 [u,v],
    \end{equation*}
    and thus \eqref{eq:basicCondition} holds.
      In the last step we have used that $T \F = \im \varphi_1$ is involutive so that $[u, \varphi_1 \sigma_1v] \in \im \varphi_1= \ker \varphi_0$.
\end{proof}

The next step is to modify the connections $D_0, \ldots, D_N$ so that we get a collection of compatible connections on $M\setminus \sf$.
Let $D=D_\End$ be the connection on $\End (E|_{M\setminus \sf})$ induced by $(D_N, \ldots, D_0, D_{basic})$.
Note that, since $\varphi_k\varphi_{k+1}=0$, it follows from \eqref{wheat}  that, for $k\geq 0$,
\begin{equation}\label{street}
  D\varphi_k \varphi_{k+1} = \varphi_kD\varphi_{k+1}.
  \end{equation}
  For $k = 0, \ldots, N$, we let
\begin{equation} \label{eq:tildeD}
    \widetilde D_k = D_k - D\varphi_{k+1}\sigma_{k+1},
\end{equation}
where by convention we set $\varphi_{N+1}$ and $\sigma_{N+1}$ to be the zero map so that $\widetilde D_N=D_N$.
Since
\begin{equation}\label{badminton}
  a_k:=- D\varphi_{k+1}\sigma_{k+1}
\end{equation}
is a smooth $1$-form on $M\setminus \sf$ with values in $\End(E_k|_{M\setminus \sf})$,
$\widetilde D_k$ is a well-defined connection on $M\setminus \sf$.
In view of \eqref{eq:Dvarphi1} and \eqref{boat}, note that
\begin{equation} \label{eq:tildeD0b0a0}
  \widetilde {D}_0 = D^{TM} + b + a_0
    \quad \quad
     \widetilde D_k=D_k  + a_k, ~~ k\geq 1.
\end{equation}

\begin{remark}\label{albin}
  For each $k\geq 0$, note that if $D_k$ and $D_{k+1}$ are $(1,0)$-connections, then $a_k$ is a $(1,0)$-form and it follows that $\widetilde D_k$ is a $(1,0)$-connection. In particular, if we assume that $D_1, \ldots, D_N$ are $(1,0)$-connections, then so are $\widetilde D_0, \ldots, \widetilde D_N$. Indeed, recall from above that $D_0$ is a $(1,0)$-connection
  since $D^{TM}$ by assumption is a $(1,0)$-connection.
\end{remark}

It remains to show that  $(\widetilde D_N, \ldots, \widetilde D_0, D_{basic})$ is compatible with the complex \eqref{eq:resolution-NF} over $M \setminus \sf$.
Let us first check the compatibility condition \eqref{eq:compatible} for $1\leq k\leq N$. Let $\beta$ be a local section of $E_k|_{M \setminus \sf}$.
Then
    \begin{align*}
 ( \widetilde{D}_{k-1} \circ \varphi_k + \varphi_k \circ  \widetilde{D}_k)\beta  &= D_{k-1}(\varphi_k \beta)-D\varphi_k\sigma_k\varphi_k \beta
        + \varphi_k D_k\beta - \varphi_k D\varphi_{k+1} \sigma_{k+1} \beta
        \\ &= D\varphi_k(I-\sigma_k\varphi_k)\beta-\varphi_k D \varphi_{k+1} \sigma_{k+1} \beta \\
        &= D\varphi_k \varphi_{k+1} \sigma_{k+1} \beta - \varphi_k D \varphi_{k+1} \sigma_{k+1} \beta = 0,
    \end{align*}
    where we have used \eqref{meat}, \eqref{eq:idSigma}, and \eqref{street}.
    To check the compatibility condition  \eqref{eq:compatible} at level $0$, let $v$ be a section of $TM|_{M \setminus \sf}$.
    Then, using that $\varphi_0\varphi_1=0$, cf.\ \eqref{ortho},
    \begin{equation*}
      D_{basic}(\varphi_0 v)=-\varphi_0 D_0(\pi_0v) = - \varphi_0 D_0v+\varphi_0 D_0(\varphi_1\sigma_1) v =
      - \varphi_0 D_0v+\varphi_0 D\varphi_1\sigma_1 v = -\varphi_0\widetilde D_0v
      \end{equation*}
      To conclude \eqref{eq:compatible} holds for each $0\leq k \leq N$ (with the convention $\widetilde D_{-1}=D_{basic}$).

     We have now proved the following.

     \begin{prop} \label{prop:basic-connection-asm}
The collection of connections $(\widetilde D_N, \ldots, \widetilde D_0, D_{basic})$ on $M\setminus \sf$  is compatible with \eqref{eq:resolution-NF} over $M \setminus \sf$.
\end{prop}

\smallskip

For future reference we notice that the connections above are almost semi-meromorphic in the following sense.

\begin{lma}\label{wedding}
  The $\End (E)$-valued forms $a_k$ and $b$ on $M\setminus \sf$,
  defined by \eqref{badminton} and \eqref{boat}, respectively,
  have continuations to $M$ as almost semi-meromorphic $\End (E)$-valued currents on $M$.
  \end{lma}

  Morally, this means that
  the $\End (E)$-valued connections $\widetilde D_k$ and $D_{basic}$ can be continued to singular connections on $M$ of the form $D+\alpha$, where $D$ is a smooth $\End (E)$-valued connection on $M$ and $\alpha$ is an almost semi-meromorphic section of $\End (E)$.

  \begin{proof}
    The main ingredient in the proof is the fact that the mappings
    $\sigma_k$ can be continued as almost semi-meromorphic sections of
    $\End (E)$, see, e.g., the proof of Lemma ~2.1 in
    \cite{LW:chern-currents-sheaves}.

    Recall that $a_0 = - D_{D_1,D_0} \varphi_1 \sigma_1$, where we have used the notation from \eqref{eq:DEndDependence}.
    By \eqref{eq:Dvarphi1} and \eqref{boat}, $D_0 = D^{TM} + b$.
    In view of \eqref{meat}, we get that
    \begin{equation}
    a_0 = -D_{D_1,D^{TM}}\varphi_1 \sigma_1 - b\varphi_1 \sigma_1.
    \end{equation}
    Furthermore, $a_k = -D_{D_{k+1},D_k} \varphi_{k+1}\sigma_{k+1}$ for $k \geq 1$.
    Since $D^{TM}, D_1,\ldots, D_N$ are (smooth) connections, the $\varphi_k$
    are holomorphic on $M$, the
    $\sigma_k$ are almost semi-meromorphic, and $ASM(M)$ is closed
    under multiplication by smooth forms, it follows that $b$ and
    $a_k$ are almost semi-meromorphic for $k\geq 0$, cf.\ Section~\ref{asmsec}.
  \end{proof}

  \subsection{The connections $\widehat D^\epsilon_k$ on $M$}\label{evening}

Let $\chi\sim\chi_{[1,\infty)}$, let $s$ be a generically nonvanishing
holomorphic section of a Hermitian vector bundle over $M$ such that $\sf
\subset \{s=0\}$, let $\epsilon>0$, let
\begin{equation}\label{winter}
\chi_\epsilon= \chi(|s|^2/\epsilon),
\end{equation}
and let $\Sigma_\epsilon$ denote the closure of $\{\chi_\epsilon < 1\}$ in $M$.
If $\chi(t)=1$ for $t\geq T$, then note that $\{\chi_\epsilon < 1\} \subset \{|s|^2 < T\epsilon\}$, so that $\Sigma_\epsilon$ is a kind of tubular neighborhood of $\sf$.
Moreover, $\bigcap_{\epsilon > 0} \Sigma_\epsilon = \{s = 0\} \supset \sf$.
\begin{remark}\label{christmas}
  If $\rho$ is the rank of $\varphi_1$ in \eqref{eq:resolution-NF}, we can choose $s$ as the section $\det^\rho \varphi_1$ of $\Lambda^\rho E_1^*\otimes \Lambda^\rho TM$; then, in fact, $\{s=0\} =\sf.$
\end{remark}

Set
\begin{equation} \label{eq:hatD}
    \widehat{D}_0^\epsilon= \chi_\epsilon \widetilde{D}_0 + (1-\chi_\epsilon) D^{TM} \quad \text{and} \quad \widehat{D}_k^\epsilon=\chi_\epsilon \widetilde{D}_k+ (1-\chi_\epsilon) D_k, \text{ for } k=1,\ldots,N,
\end{equation}
where the $\widetilde{D}_k$ are the connections defined in
\eqref{eq:tildeD}. Note that
$\widehat{D}_N^\epsilon,\dots,\widehat{D}_0^\epsilon$ are connections
on $M$, and that $\widehat{D}_k^\epsilon = \widetilde D_k$ in
$M\setminus \Sigma_\epsilon$ for $k=0,\ldots, N$.
 Since   $(\widetilde D_N, \ldots, \widetilde D_0, D_{basic})$ is compatible with \eqref{eq:resolution-NF} in $M\setminus \sf$ by Proposition ~\ref{prop:basic-connection-asm}, it follows as in Section ~\ref{lucia} that if $\Phi\in \C[z_1,\ldots, z_n]$ is a homogeneous symmetric polynomial of degree $\ell$ with $n-\kappa<\ell\leq n$, then
      \begin{equation}\label{parkman}
    \left (\frac{i}{2\pi}\right )^\ell \Phi\big (\Theta(\widehat{D}_N^\epsilon) | \ldots | \Theta(\widehat{D}_0^\epsilon)\big )
    \end{equation}
    is a closed form of degree $2\ell$ with support in
        $\Sigma_\epsilon$.

\begin{remark}\label{axel}
  Note in view of Remark ~\ref{albin} that $\widehat D^\epsilon_k$ is
  a $(1,0)$-connection if $D_k$ and $D_{k+1}$ are. Also recall from
  above that $D_0$ is a $(1,0)$-connection since $D^{TM}$ is.
  Hence, if we assume that $D_1,\ldots, D_N$ are $(1,0)$-connections, then so are $\widehat D^\epsilon_0, \ldots, \widehat D^\epsilon_N$.
  \end{remark}

\smallskip

Let us fix a compact connected component $Z$ of $\sf$. Then, after
possibly shrinking $M$, we may assume that $\beta:=(M, (E_N, \ldots,
E_1, TM), (\varphi_N, \ldots, \varphi_0))$ is a $Z$-sequence. In
particular, $\sf=Z$ and thus
 we can choose $s$ so that
$\{s=0\} =Z,$ cf.\ Remark ~\ref{christmas}.
Then, for $\epsilon$ sufficiently small, $\Sigma_\epsilon$ is a compact neighborhood of $Z$.
Since $(\widehat{D}_N^\epsilon,\dots,\widehat{D}_0^\epsilon,D_{basic})= (\widetilde{D}_N^\epsilon,\dots,\widetilde{D}_0^\epsilon,D_{basic})$ is compatible with the  complex \eqref{eq:resolution-NF} in $M\setminus \Sigma_\epsilon$ by Proposition ~\ref{prop:basic-connection-asm} and $D_{basic}$ is a basic connection on $N \F|_{M \setminus Z}$ by Lemma ~\ref{lemma:Dbasic}, it follows that
$(\widehat{D}_N^\epsilon,\dots,\widehat{D}_0^\epsilon,D_{basic})$  is fitted to
$\beta$
for $\epsilon$ sufficiently small. In fact, one may check
that the construction of fitted connections in \cite{baum-bott} with some
minor adaptation agrees with
$(\widehat{D}_N^\epsilon,\dots,\widehat{D}_0^\epsilon,D_{basic})$.  Now in view of the definition of $\res^\Phi(\F;Z)$, see Section ~\ref{lucia}, we get the following.

 \begin{lma}\label{lussekatt}
   Assume that $(M, (E_N, \ldots, E_1, TM), (\varphi_N, \ldots, \varphi_0))$ is a $Z$-sequence and that $s$ is a holomorphic section of a Hermitian vector bundle such that $\{s=0\}=Z$.
   Then, for sufficiently small $\epsilon>0$, the form \eqref{parkman} represents the Baum-Bott residue $\res^\Phi(\F;Z) \in H_{2n-2\ell}(Z,\C)$.
\end{lma}

\section{Baum-Bott currents} \label{sec:BB-currents}

In this section we prove Theorem~ \ref{thm:BB-current}. We prove that
the limits of \eqref{parkman} as $\epsilon\to 0$ exist as
pseudomeromorphic currents with support on $S=\sing \F$; we call these
\emph{Baum-Bott (residue) currents}. If $Z$ is a compact connected
component of $S$, then the restriction to $Z$ represents the corresponding Baum-Bott residue.

We have the following more precise version of Theorem
~\ref{thm:BB-current}; for the part
about the independence of the choice of metrics and connections, see Corollary ~\ref{famine} below.

\begin{thm}\label{train}
  Let $M$ be a complex manifold of dimension $n$, let $\F$ be a  holomorphic foliation of rank $\kappa$ on $M$, and let $\Phi \in \C[z_1,\ldots,z_n]$ be a homogeneous symmetric polynomial of degree $\ell$ with $n-\kappa<\ell\leq n$.
  Assume that the normal sheaf $N\F$ of $\F$ admits a locally free
  resolution of the form \eqref{eq:resolution-NF}. Moreover, assume
  that $TM, E_1, \ldots, E_N$ are equipped with Hermitian metrics, and
  connections $D^{TM}, D_1, \ldots, D_N$, respectively, and assume
  that $D^{TM}$ is of type $(1,0)$ and torsion free. Let
  $\chi\sim\chi_{[1,\infty)}$ and let $s$ be a generically
  nonvanishing holomorphic section of a Hermitian vector bundle over
  $M$ such that $\{s=0\}\supset \sing \F$, and let $\widehat
  D^\epsilon_0, \ldots, \widehat D^\epsilon_N$ be the connections defined by \eqref{eq:hatD}.

  Then
      \begin{equation}\label{morning}
      R^\Phi:= \lim_{\epsilon \to 0} \left (\frac{i}{2\pi}\right )^\ell \Phi\big(\Theta(\widehat D_N^\epsilon) | \ldots | \Theta(\widehat D_0^\epsilon)\big )
    \end{equation}
     is a well-defined closed
     pseudomeromorphic current on $M$ of degree $2\ell$ with support on
     $\sing \F$. Moreover $R^\Phi$ only depends on the complex
     \eqref{eq:resolution-NF} and the Hermitian metrics and
     connections $D^{TM}, D_1, \ldots, D_N$ close to $\sing\F$, and in particular
     is independent of the choice of $\chi$ and $s$.
If we assume that also $D_1, \ldots, D_N$ are of type $(1,0)$, then $R^\Phi$ is a sum of currents of bidegree $(\ell+j, \ell-j)$ for $0\leq j\leq \ell$.

   Let $Z$ be a connected component of $\sing \F$ and let
   \begin{equation}\label{photo}
     R^\Phi_Z=\1_Z R^\Phi.
     \end{equation}
 If $Z$ is compact, then $R^\Phi_Z$ represents $\res^\Phi(\F; Z)$.
\end{thm}

\begin{proof}
  We partially follow the proof of
  \cite[Theorem~ 5.1]{LW:chern-currents-sheaves}.

  We first prove that the limit \eqref{morning} exists and is a pseudomeromorphic current. This is a local statement and we may therefore work in a local trivialization. Let $\theta^{TM}, \theta_1, \ldots, \theta_N$ be the connection matrices for $D^{TM}, D_1, \ldots, D_N$, respectively. Then the connection matrices for $\widehat D^\epsilon_k$ are given by
  \begin{equation}\label{secret}
    \widehat \theta^\epsilon_0= \theta^{TM}+\chi_\epsilon (b+a_0), \quad \quad
     \widehat\theta^\epsilon_k=\theta_k + \chi_\epsilon a_k, ~~ k\geq
     1,
\end{equation}
cf.\ \eqref{eq:tildeD0b0a0} and \eqref{eq:hatD}, where $b$ and $a_k$ are defined by \eqref{boat} and \eqref{badminton}, respectively.
 It follows that for $k\geq 1$ the curvature matrix equals
    \begin{equation*}
      \widehat \Theta_k^\epsilon= d \widehat \theta^\epsilon_k +(\widehat \theta^\epsilon_k)^2 =
      \Theta_k +  d(\chi_\epsilon a_k)+ \theta_k\wedge \chi_\epsilon a_k + \chi_\epsilon a_k \wedge \theta_k + \chi_\epsilon^2 a_k\wedge a_k,
      \end{equation*}
      and similarly
       \begin{equation*}
      \widehat \Theta_0^\epsilon=
      \Theta^{TM} +  d(\chi_\epsilon (b+a_0))+ \theta^{TM} \wedge \chi_\epsilon (b+a_0) + \chi_\epsilon (b+a_0) \wedge \theta^{TM} + \chi_\epsilon^2 (b+a_0)\wedge (b+a_0).
      \end{equation*}
      Since $b$ and $a_k$ are smooth forms in $M\setminus \sf$ that have continuations to $M$ as almost semi-meromorphic currents by Lemma ~\ref{wedding}, and $ASM(M)$ is an algebra, cf.\ Section ~\ref{asmsec}, it follows that, for $k=0, \ldots, N$, $\widehat \Theta^\epsilon_k$ is a matrix-valued form of the form
      \begin{equation}\label{bluewear}
   \widehat \Theta_k^\epsilon= \alpha_k+ \chi_\epsilon \beta_k' + \chi_\epsilon^2 \beta_k'' + d\chi_\epsilon \wedge \beta_k''',
   \end{equation}
   where $\alpha_k, \beta_k', \beta_k''$, and $\beta_k'''$ are
   independent of $\chi_\epsilon$, $\alpha_k$ is smooth,  and $\beta_k', \beta_k''$, and $\beta_k'''$ are almost semi-meromorphic currents that are smooth in $M\setminus \sf$. To see this, note that if $a\in ASM(M)$ is smooth outside $\sf$, then $\chi_\epsilon da = \chi_\epsilon \1_{M\setminus \sf} da$ and since $\1_{M\setminus \sf} da$ is almost semi-meromorphic, see Lemma ~\ref{gold}, then $\chi_\epsilon da$ is of the form $\chi_\epsilon \beta_k'$.

In view of \eqref{eq:mixed-chern}, note that $\Phi(\Theta(\widehat D_N^\epsilon) | \ldots | \Theta(\widehat D_0^\epsilon))$ is a polynomial in the entries of $\widehat \Theta_k^\epsilon$, $k=0 \ldots, N$. Since $ASM(M)$ is an algebra and $d\chi_\epsilon\wedge d\chi_\epsilon=0$, it follows that $\Phi(\Theta(\widehat D_N^\epsilon) | \ldots | \Theta(\widehat D_0^\epsilon))$ is of the form
\begin{equation}\label{whitecollar}
    A+ \sum_{j\geq 1} \chi_\epsilon^j B_j'+ \sum_{j\geq 1} \chi_\epsilon^{j-1} d\chi_\epsilon \wedge B_j'',
    \end{equation}
    where $A, B_j'$ and $B_j''$ are independent of $\chi_\epsilon$, $A$ is smooth, and $B_j'$ and $B_j''$ are almost semi-meromorphic currents that are smooth in $M\setminus \sf$.
    We conclude
    in view of Remarks ~\ref{elsa} and ~\ref{bronze}, that the limit
    as $\epsilon \to 0$ of each term in \eqref{whitecollar} exists as
    a pseudomeromorphic current and the limit is independent of the
    choice of $\chi$ and $s$. Hence \eqref{morning} is a well-defined
    pseudomeromorphic current independent of the choice of $\chi$ and
    $s$.
    Since it is independent of $s$ we may assume that $\{s=0\}=\sf$, cf.\ Remark ~\ref{christmas}. Then,  since \eqref{parkman} has support in $\Sigma_\epsilon$, see Section ~\ref{lucia}, it follows that
    $R^\Phi$ is a closed current of degree $2\ell$ with support in $\bigcap_{\epsilon > 0} \Sigma_\epsilon=\sf$.

Note that the connections $\widehat D^\epsilon_k$ are locally defined,
in the sense that on any open set $U$, the  $\widehat D^\epsilon_k$ only
depend on \eqref{eq:resolution-NF}, the Hermitian metrics, and the
connections $D^{TM}, D_1, \ldots, D_N$ on $U$, cf.\ Section
~\ref{sec:basic-connection}. It follows that $R^\Phi$ is locally
defined in the same sense.

Assume now that $D_1, \ldots, D_N$ are $(1,0)$-connections. Then so are  $\widehat D_{0}^{\epsilon}, \ldots, \widehat D_N^\epsilon$, see Remark ~\ref{axel}.
 Thus each $\widehat \Theta^{\epsilon}_{k}$ has components of bidegree $(1,1)$ and $(2,0)$. It follows that $\Phi\big(\Theta(\widehat D_N^\epsilon) | \ldots | \Theta(\widehat D_0^\epsilon)\big)$ only has components of
 bidegree $(\ell + j, \ell-j)$, $0\leq j\leq \ell$.
 Hence so has the limit $R^\Phi $.

\smallskip

Now let $Z$ be a compact connected component of $\sf$. Since $R^\Phi$
is locally defined, after possibly shrinking $M$ we may assume that
$(M, (E_N, \ldots, E_1, TM), (\varphi_N, \ldots, \varphi_0))$ is a
$Z$-sequence; then $Z$ is the only connected component of $\sf$ and
thus $R^\Phi_Z=R^\Phi$, cf.\ \eqref{photo}.
Since $R^\Phi$ is independent of $s$ we can choose $s$ so that
$\{s=0\}=Z$, cf.\ Remark ~\ref{christmas}. Now, by Lemma
~\ref{lussekatt}, the form \eqref{parkman} represents $\res^\Phi(\F; Z)$ for all $\epsilon>0$ sufficiently small.  Thus so does the limit $R^\Phi_Z$ by Poincar\'e duality.
\end{proof}

\subsection{Dependence on the metrics and connections}\label{homotopy}

The following result gives a description of how the
Baum-Bott currents depend on the choice of metrics and connections.

\begin{prop} \label{prop:changeConnectionsMetrics}
    Let $M$, $\F$, $\Phi$, and \eqref{eq:resolution-NF} be as in
    Theorem ~\ref{train}.
  For each $j=1,2$, assume that $TM, E_1, \ldots, E_N$ are equipped with Hermitian metrics and connections $D^{TM}_{(j)}, D_1^{(j)}, \ldots, D_N^{(j)}$, such that $D^{TM}_{(j)}$ is of type $(1,0)$ and torsion free, and let $R^\Phi_{(j)}$ denote the corresponding Baum-Bott current \eqref{morning}.
  Then there exists a pseudomeromorphic current $N^\Phi$ of degree
  $2\ell -1$ with support on $\sing \F$ such that
    \begin{equation}\label{oberwolfach}
  d N^\Phi=  R^\Phi_{(1)}-R^\Phi_{(2)}.
\end{equation}
    Furthermore, if also $D^{(j)}_1,\ldots, D^{(j)}_N$ are of type $(1,0)$, then $N^\Phi$ is a sum of currents of bidegree
    $(\ell+j,\ell-1-j)$ for $0\leq j \leq \ell-1$.
  \end{prop}

  \begin{proof}
    Let $\sigma_k^{(1)}$ and  $\sigma_k^{(2)}$ denote the minimal inverses of $\varphi_k$ with respect to the two different choices of Hermitian metrics.
    Next, for $j=1,2$ and $k=0,\ldots, N$, let $\widehat D^{(j),\epsilon}_{k}$ be the connection  \eqref{eq:hatD} constructed in Section ~\ref{sec:basic-connection} from the connections $D^{TM}_{(j)}, D_1^{(j)}, \ldots, D_N^{(j)}$ and the minimal inverses $\sigma_1^{(j)}, \ldots, \sigma_N^{(j)}$.

    Following the proof of Proposition ~5.31 in \cite{baum-bott}, let $\widetilde M = M\times [0,1]$
    and let $\pi:\widetilde M\to M$ be the natural projection. Next, for
$t \in [0,1]$, $k=0,\ldots, N$, define
$$\widehat{D}^\epsilon_{t, k} := t \widehat{D}^{(1),\epsilon}_{k}  + (1-t) \widehat{D}^{(2),\epsilon}_{k},$$ where $\widehat D^{(j),\epsilon}_{k}$ now denote the pullback connections on $\pi^*E_k$,
and let
\begin{equation*}
N^\Phi_\epsilon =\left (\frac{i}{2\pi}\right )^\ell  \pi_* \Phi\big(\Theta(\widehat D_{t,N}^\epsilon) | \ldots | \Theta(\widehat D_{t,0}^\epsilon)\big).
  \end{equation*}
  Then, by (the proof of) Proposition ~5.31 in \cite{baum-bott}, $N^\Phi_\epsilon$ is a form of degree $2\ell-1$ with support in $\Sigma_\epsilon$, such that
  \begin{equation}\label{hildesheim}
    dN^\Phi_\epsilon =
    \left (\frac{i}{2\pi}\right )^\ell  \Phi\big(\Theta(\widehat D_{N}^{(1),\epsilon}) | \ldots | \Theta(\widehat D_{0}^{(1),\epsilon})\big)
    -
    \left (\frac{i}{2\pi}\right )^\ell  \Phi\big(\Theta(\widehat D_{N}^{(2),\epsilon}) | \ldots | \Theta(\widehat D_{0}^{(2),\epsilon})\big).
  \end{equation}

  To prove existence of the limit of $N^\Phi_\epsilon$, we may as in the proof of
  Theorem ~\ref{train} work in local chart.
  Since the $\sigma_k^{(j)}$ are almost semi-meromorphic, see
  the proof of Proposition ~\ref{prop:basic-connection-asm}, as in the
  proof of Theorem ~\ref{train} we get that the curvature forms
  $\widehat \Theta^{(j),\epsilon}_{k}$ corresponding to the $\widehat
  D_k^{(j), \epsilon}$ are of the form \eqref{bluewear}. Moreover,
  since $\Phi(\Theta(\widehat D_{t,N}^\epsilon) | \ldots |
  \Theta(\widehat D_{t,0}^\epsilon))$ is a polynomial in the entries
  of $\widehat \Theta^{\epsilon}_{t,k}$ it follows that $\Phi(\Theta(\widehat D_{t,N}^\epsilon) | \ldots |
  \Theta(\widehat D_{t,0}^\epsilon))$ is a polynomial in $t$ and $dt$ with coefficients of the form
  \eqref{whitecollar}. Hence, integrating with respect to $t$, it follows that
  $N^\Phi_\epsilon$ is of the form \eqref{whitecollar},
  and as in the proof of Theorem~\ref{train}, it follows that the
  limit of $N^\Phi_\epsilon$ as $\epsilon\to 0$ exists as a
  pseudomeromorphic current $N^\Phi$ of degree $2\ell-1$. As before we
  may assume that $\{s=0\}=S$ and, since $N^\Phi_\epsilon$ has support
  in $\Sigma_\epsilon$, it follows that $N^\Phi$ has support on $S$.
  Taking limits in \eqref{hildesheim} we get \eqref{oberwolfach}.

  \smallskip

  Assume now that $D_1, \ldots, D_N$ are $(1,0)$-connections. Then the connection matrices of $\widehat D_{0}^{\epsilon}, \ldots, \widehat D_N^\epsilon$
  in frames induced by local holomorphic frames on $E_0,\ldots,E_n$ consist of (first degree) polynomials in $t$ with coefficients that are forms on $X$ of bidegree $(1,0)$, cf., Remark~\ref{axel}.
  Thus the  $\widehat \Theta^{(j),\epsilon}_{k}$ are polynomials in $t$ and $dt$ with coefficients that are forms on $X$ of bidegree $(1,0)$, $(1,1)$ and $(2,0)$,
  and the terms which contain $dt$ are exactly the ones where the coefficients are forms of bidegree $(1,0)$. It follows that
  \begin{equation*}
    \Phi\big(\Theta(\widehat D_{t,N}^\epsilon) | \ldots | \Theta(\widehat D_{t,0}^\epsilon)\big)=\Phi^\epsilon_0 + \Phi^\epsilon_1\wedge dt,
    \end{equation*}
    where $\Phi^\epsilon_0$ is a $2\ell$-form with no occurrences of $dt$ and $\Phi^\epsilon_1$ is a polynomial in $t$ with coefficients that are $(2\ell-1)$-forms on $X$
    with components of bidegree $(\ell + j, \ell-1-j)$, $0\leq j\leq \ell-1$.
    Hence
    $$
    N^\Phi_\epsilon = \left (\frac{i}{2\pi}\right )^\ell  \pi_* (\Phi^\epsilon_1\wedge dt)
    $$
    has components of bidegree $(\ell + j, \ell-1-j)$, $0\leq j\leq \ell-1$
    and consequently so has $N^\Phi$.
    \end{proof}

    From Proposition ~\ref{prop:changeConnectionsMetrics} we get that
    $R^\Phi_Z$ is canonical in the following sense when $\codim Z\geq \deg \Phi$.
    \begin{cor}\label{famine}
      Assume that we are in the setting of Theorem ~\ref{train} and
      that in addition $D_2,\ldots, D_N$ are of type $(1,0)$. Let $Z$
      be a connected component of $\sing \F$. Assume that $\codim
      Z\geq \ell$. Then $R^\Phi_Z$ is independent of the choice of
      Hermitian metrics and connections on $TM, E_1, \ldots, E_N$.
  \end{cor}

  \begin{proof}
    Let $R^\Phi_{Z, (j)}$,  $j=1,2$,  denote the Baum-Bott currents corresponding to two different choices of metrics and connections.
    Then, by Proposition ~\ref{prop:changeConnectionsMetrics},
$$R^\Phi_{Z, (1)}- R^\Phi_{Z, (2)}=\1_Z d N^\Phi,$$
where $N^\Phi$ is a pseudomeromorphic current with components of
bidegree $(*, q)$ with $q\leq \ell-1$, cf.\ \eqref{photo}. Let $U$ be a neighborhood of $Z$
containing only the connected component $Z$ of $\sf$.
Since $N^\Phi|_U$ has support on $Z$ of codimension $\geq \ell$, it
follows from the dimension principle (Proposition \ref{prop:dimension-principle}) that
$N^\Phi|_U$ vanishes, and consequently so does
$dN^\Phi|_U$. Furthermore, since $\1_Z d N^\Phi$ only depends on
$dN^\Phi|_U$, it follows that $\1_Z dN^\Phi$ vanishes. Thus
$R^\Phi_{Z, (1)}= R^\Phi_{Z, (2)}$, which proves the result.
\end{proof}

\subsection{Dimension of singular components} \label{sec:dim-sing}
The classical Baum-Bott theory naturally yields restrictions on the
dimension of the singular components of a given foliation once we can
show that certain Baum-Bott classes are not vanishing.
More
precisely, if $Z$ is a singular component of a
given foliation $\F$ on a compact manifold $M$ such that
$\Res^\Phi(\F;Z)$ is non-zero for some homogeneous symmetric
polynomial $\Phi$ of degree $\ell$, then $\dim Z \geq n - \ell$.
This can be easily obtained from Baum-Bott's theory,  since $\Res^\Phi(\F;Z)$ lies in the image of the natural map $\iota_* : H_{2n-2\ell}(Z,\C) \to H_{2n-2\ell}(M,\C)$ (see Section ~\ref{lucia}). For example, this automatically gives  lower bounds on the dimension of singular components of foliations on projective spaces. This is because  a foliation  $\F$ of rank $\kappa$ and degree $d$ on $M=\mathbb{P}^n$ satisfies $c_1(N\F) = (d+n-\kappa+1) H$, where $H$ is the class of a hyperplane, cf., i.e., \cite[p.\ ~6]{CM}. In particular $c_1^{n-\kappa+1}(N\F) \neq 0 $, so $\Res^{c_1^{n-\kappa+1}}(\F;Z) \neq 0$ for at least one singular component $Z$, yielding $\dim Z \geq  \kappa -1$.

The above result can be slightly improved in our setting,  using the dimension principle for pseudomeromorphic currents.   An advantage of Corollary \ref{cor:dim-Z} below is that it can in principle be applied in local settings when, for instance, the homology of $M$ is trivial. This could help answering questions on the existence of certain singular foliations.  For example,  it is an open question whether there is a germ of codimension two foliation with an isolated singularity at the origin of $\C^4$,  see \cite[Problem 1]{CLN:JDG}.

\begin{cor} \label{cor:dim-Z}
   Assume that we are in the setting of Theorem ~\ref{train} and
      that in addition $D_2,\ldots, D_N$ are of type $(1,0)$. Let $Z$
      be a connected component of $\sing \F$. Assume that $R_Z^\Phi
      \neq 0$. Then $\dim Z \geq n - \ell$.
\end{cor}
\begin{proof}
  By Theorem 6.1.,
  $R_Z^\Phi = \sum_{j=0}^\ell R_{\ell-j}$,  where $R_{\ell-j}$ is a
  pseudomeromorphic current of bidegree $(\ell + j,  \ell -j)$
  supported by $Z$.  Notice that,  since $R_Z^\Phi \neq 0$ by
  assumption,  $R_p \neq 0$ for some $p, 0\leq p\leq \ell$.  By the dimension principle (Proposition \ref{prop:dimension-principle}), we must have $\codim Z \leq p \leq \ell$,  yielding the result.
\end{proof}

    \section{Baum-Bott currents of holomorphic vector fields} \label{sec:vector-field}

Let us consider the situation when
$\F$ is a rank one foliation on $M$.
Since $\F$ is a subsheaf of $TM$, it is torsion-free.
    Then, it follows by i.e., \cite[Lemmas 1.1.12, 1.1.15 and 1.1.16]{oss-vb} that
    $\F$ being saturated implies that $L:=T\F$ is a line bundle and $\F$ defines a global section $X \in H^0(M,TM\otimes
    L^*)$. Note that $\sing \F=\{X=0\}$.
    In particular, seeing $X$ as a morphism $L \to TM$, we obtain a locally free resolution
form
\begin{equation} \label{eq:resolution-vector-field}
    0 \to L \stackrel{X}{\to} TM \to N\F \to 0
\end{equation}
of $N \F$.

In this section we give an explicit description, Corollary
~\ref{nordstan}, of the Baum-Bott currents $R^\Phi_{\{p\}}$ when
$p$ is an isolated singularity.
We first consider the case when $\F$ is given by a global vector field
$X$, not necessarily with isolated singularities. Then the line bundle $L$ is trivial and the map $\Ok_M \to TM$ is given by multiplication by $X$.
We show that in this case the Baum-Bott currents $R^\Phi$
can be expressed in terms of the residue current of Bochner-Martinelli
type associated with $X$, cf.\ Section ~\ref{bmx}, see Theorem
\ref{thm:BBBM} below.
In particular, when $X$ has isolated singularities, we recover the
usual Baum-Bott formula in terms of the Grothendieck residue, cf.\  \eqref{eq:grothendieck-residue}.

\smallskip

Let $\Omega= (TM)^*$.
Then $X$ is a section of the dual bundle $\Omega^*$. Assume that $TM$ is equipped with a Hermitian metric and equip $\Omega$ with the dual metric.
Let $\sigma$ be the minimal inverse of $X$ and let
\begin{equation*}
    R^X =\sum_k R^X_k := \lim_{\epsilon \to 0} \dbar\chi_\epsilon \wedge \sum_{\ell\geq 0} \sigma(\dbar\sigma)^\ell,
  \end{equation*}
  where $\chi_\epsilon$ is as in \eqref{winter},
  be the residue current of Bochner-Martinelli type as defined in Section ~\ref{bmx}.

By the natural isomorphism
$$\Omega= (TM)^* \cong  T^* M $$
a section of $\Omega$ can be regarded as a $(1,0)$-form. Let $\widetilde \sigma$ denote the  form corresponding to $\sigma$.
Then note that
\begin{equation} \label{eq:tildeSigma}
    \widetilde{\sigma} = \sigma\left(dz\cdot\frac{\partial}{\partial z}\right),
\end{equation}
where $dz\cdot\frac{\partial}{\partial z}$ is the canonical $TM$-valued $(1,0)$-form \eqref{heart}.
Similarly sections of $\Lambda(\Omega\oplus T^*_{0,1}(M))$ are naturally
identified with forms,
cf.\ Section ~\ref{bmx}.
Let $\widetilde R^X_k$ denote the pseudomeromorphic $(k,k)$-current corresponding to $R^X_k$.
Then
\begin{equation} \label{eq:tildeRX}
    \widetilde{R}^X_k = \lim_{\epsilon \to 0} \dbar\chi_\epsilon \wedge \widetilde{\sigma}\wedge (\dbar\widetilde{\sigma})^{k-1}.
  \end{equation}

 \begin{ex}\label{grothen}
   Assume that $\{X=0\}$ consists of the point $p\in M$.
Let $(z_1,\ldots, z_n)$ be a local coordinate system centered at
$p$. Then $X$ is of the form
$X=\sum_{i=1}^n a_i(z) \frac{\partial}{\partial z_i}$
near $p$.
It follows from Example ~\ref{complete} that
  \begin{equation}\label{halmstad}
    \widetilde R^X=\widetilde R^X_n =
    \dbar \frac {1}{a_n}\wedge \cdots \wedge \dbar \frac{1}{a_1} \wedge dz_1\wedge \cdots \wedge dz_n.
  \end{equation}
   \end{ex}

\smallskip

To describe the Baum-Bott currents in this case we need to introduce
some notation.
As above, see Section \ref{texas}, we identify a homogeneous
symmetric polynomial $\Phi \in \C[z_1,\ldots,z_n]$ of degree $\ell$
with the corresponding invariant symmetric polynomial on (form-valued) $(n\times n$)-matrices.
Recall that the polarization of $\Phi$ is the invariant symmetric function $\widetilde \Phi(A_1,\dots,A_\ell)$
 that satisfies $\Phi(A) = \widetilde \Phi(A,\dots,A)$.
For $0 \leq k \leq \ell$, and two $(n\times n)$-matrices $A$ and $B$ (that are possibly form-valued of even degree), we let
\begin{equation}\label{liseberg}
\Phi_k(A,B) = \binom{\ell}{k} \widetilde \Phi(\underbrace{A,\dots,A}_{\text{$k$ times}},\underbrace{B,\dots,B}_{\text{$\ell-k$ times}}),
  \end{equation}
so that
\begin{equation} \label{eq:PhiAdd}
    \Phi(A+B) = \sum_{0 \leq k \leq \ell} \Phi_k(A,B).
  \end{equation}

In the statement below, $\Dc$ is as in \eqref{eq:Dvarphi1}.
\begin{thm} \label{thm:BBBM}
    Let $M$ be a complex manifold of dimension $n$ and let $X$ be a holomorphic vector field on $M$. Let $\F$ be the corresponding rank one foliation and
    consider the resolution
\begin{equation} \label{eq:resolution-vector-field-2}
    0 \to \Ok_M \stackrel{X} {\to} TM \to N\F \to 0.
  \end{equation}
    Assume that $TM$ and $\Ok_M$ are equipped with Hermitian metrics and  $(1,0)$-connections $D^{TM}$ and $D_1$, respectively, and assume that $D^{TM}$ is torsion free.
Let $\Phi\in \C[z_1,\ldots, z_n]$ be a homogeneous symmetric polynomial of degree $n$, and
  let $R^\Phi$ denote the associated Baum-Bott current \eqref{morning}. Then there exists a pseudomeromorphic $(n,n-1)$-current $N^\Phi$ with support on $\{X=0\}$ such that
    \begin{equation} \label{eq:BB-current-vector-field}
        R^\Phi= \left (\frac{i}{2\pi}\right )^n  \sum_{k = \codim
          \{X=0\}}^n \widetilde{R}_k^X \wedge \Phi_k\big (\Dc
        \varphi_1,\Theta(D^{TM})\big ) + \dbar N^\Phi.
    \end{equation}
\end{thm}

\begin{proof}
  Let us use the notation from the previous sections (with the convention $E_1=\Ok_M$). In particular, let $\varphi_1: \Ok_M \to TM$
  be the map given by multiplication by $X$ and let $\sigma_1$ be its minimal inverse.
  By Proposition~\ref{prop:changeConnectionsMetrics}, we may assume
  that $\Ok_M$ is equipped with the trivial metric and $D_1 = d$ is
  the trivial connection, because these choices will only affect the
  term $\dbar N^\Phi$ in \eqref{eq:BB-current-vector-field}. Indeed,
  since $D_1$ and $d$ are $(1,0)$-connections the difference between
  the corresponding Baum-Bott currents is of the form $dN^\Phi$, where $N^\Phi$ is a pseudomeromorphic current of bidegree $(n,n-1)$.

Let $\widehat{D}_0^\epsilon$ and $\widehat{D}_1^\epsilon$  be the connections defined in \eqref{eq:hatD}. Then, by definition,
\begin{equation*}
  R^\Phi= \lim_{\epsilon\to 0}\left (\frac{i}{2\pi}\right )^n 
 \Phi\big (\Theta(\widehat{D}_1^\epsilon) | \Theta(\widehat{D}_0^\epsilon)\big ).
\end{equation*}
  Observe that, in the notation of the previous sections, we have that  $\varphi_2 = 0$. Therefore, $\widehat{D}_1^\epsilon = D_1 = d$, see \eqref{eq:tildeD} and \eqref{eq:hatD}; in particular, $\Theta(\widehat{D}_1^\epsilon) = 0$. Hence
\begin{equation*}
  R^\Phi= \lim_{\epsilon\to 0} \left (\frac{i}{2\pi}\right )^n
  \Phi\big (\Theta(\widehat{D}_0^\epsilon)\big ),
\end{equation*}
cf.\ Section ~\ref{texas}.
    Now \eqref{eq:BB-current-vector-field} follows by combining Lemmas ~\ref{step1} and ~\ref{step2} below.
    \end{proof}

Let
\begin{equation} \label{eq:Dvarphi1-epsilon}
  D_0^\epsilon = \chi_\epsilon D_0 + (1 - \chi_\epsilon)D^{TM},
\end{equation}
where $D_0$ is as in \eqref{eq:Dvarphi1}.
\begin{remark}\label{newyear}
Since $D^{TM}$ is a $(1,0)$-connection on $M$ and $D_0$ is a $(1,0)$-connection on $M\setminus \sf$, $D_0^\epsilon$ is a $(1,0)$-connection on $TM$.
Note that $D_0^\epsilon=D_0$ on $M\setminus \Sigma_\epsilon$.
\end{remark}

    \begin{lma}\label{step1}
      Assume that $D_0^\epsilon$ and $\widehat{D}_0^\epsilon$ are the
      connections on $TM$ defined by \eqref{eq:Dvarphi1-epsilon} and
      \eqref{eq:hatD}, respectively.
      Then there exists a pseudomeromorphic $(n,n-1)$-current $N^\Phi$ with support on $\{X=0\}$ such that
  \begin{equation}\label{betapet}
      \lim_{\epsilon\to 0} \Phi\big
      (\Theta(\widehat{D}_0^\epsilon)\big )= \lim_{\epsilon\to 0}
      \Phi\big (\Theta({D}_0^\epsilon)\big ) + \dbar N^{\Phi}.
    \end{equation}
  \end{lma}

  \begin{lma}\label{step2}
     Assume that $D_0^\epsilon$ is the connection on $TM$ defined by \eqref{eq:Dvarphi1-epsilon}. Then
 \begin{equation}\label{mantel}
    \Phi\big (\Theta(D_0^\epsilon)\big ) =
    \sum_{k = \codim Z}^n \widetilde{R}_k^X \wedge \Phi_k\big (\Dc \varphi_1,\Theta(D^{TM})\big ).
  \end{equation}
\end{lma}

To prove the lemmas we recall from \cite[\S 8]{baum-bott} that, if $X$ is a non-vanishing vector field on some open set $U \subset M$, a connection $D$ on $TM|_U$ is called an \textit{$X$-connection} if $D$ is of type $(1,0)$ and if
\begin{equation} \label{eq:X-connection-def}
i(X) D Y =  [X,Y]
\end{equation}
 for every vector field $Y$ on $U$.

 The following result is the analogue of Theorem ~\ref{thm:vanishing} for $X$-connections.

 \begin{lma} \cite[Lemma 8.11]{baum-bott} \label{lemma:X-vanishing}
Let $X$ be a non-vanishing holomorphic vector field on $U \subset M$ and let $D$ be an $X$-connection on $TM|_U$. Then $\Phi(\Theta(D)) = 0$ for any homogeneous symmetric polynomial $\Phi \in \C[z_1,\ldots,z_n]$ of degree $n$.
\end{lma}

\begin{lma}\label{roadtrip}
  The connection $D_0$ defined in \eqref{eq:Dvarphi1} is an $X$-connection on $TM|_{M \setminus \{X=0\}}$.
 \end{lma}

\begin{proof}(Compare to Lemma ~\ref{lemma:Dvarphi1}.)
By assumption $X$ is non-vanishing in $M\setminus \{X=0\}$, so that the statement makes sense.

We saw in Section ~\ref{pomona} that $D_0$ is a $(1,0)$-connection. It remains to prove that it satisfies \eqref{eq:X-connection-def} in $M\setminus \{X=0\}$.
Since  $\sigma_1$ is the inverse of $\varphi_1$ in $M\setminus \{X=0\}$, $\sigma_1X=1$ there. Thus, by \eqref{afternoon}
\begin{equation*}
  i(X) \big (\Dc \varphi_1\sigma_1 (dz\cdot \partial \slash \partial z) \big ) =  \Dc \varphi_1\sigma_1 X =
 \Dc \varphi_11
\end{equation*}
in $M\setminus \{X=0\}$.
Now, by \eqref{abba}, \eqref{jewel}, and the fact that $D_11=d1=0$,
\begin{equation*}
  \Dc \varphi_11(Y)= \Dc \varphi_1 (Y\otimes 1) = - i(Y) D_{\End}\varphi_1 1 = - i(Y) D^{TM}(\varphi_1 1) =-i(Y) D^{TM} X;
\end{equation*}
here $D_{\End}$ is the connection on $\End (E)$ induced by $D_1$ and $D^{TM}$.
Hence
\begin{equation*}
   i(X)D_0Y = i(X)D^{TM} Y + i(X) \big (\Dc \varphi_1\sigma_1 (dz\cdot \partial \slash \partial z) \big ) Y =
   i(X) D^{TM}  Y - i(Y) D^{TM}X = [X,Y],
   \end{equation*}
  where the last equality follows since $D^{TM}$ is torsion free, cf.\  \eqref{eq:torsion-free-def}.
    \end{proof}

  \begin{proof}[Proof of Lemma ~\ref{step1}]
    As in the proof of Proposition ~\ref{prop:changeConnectionsMetrics} let
   $\widetilde M=M\times[0,1]$ and let $\pi:\widetilde M\to M$ be the natural projection.
    Let
    \begin{equation*}
      D^\epsilon_t=t\widehat D_0^\epsilon + (1-t)D^\epsilon_0,
      \end{equation*}
      where now $\widehat D_0^\epsilon$ and $D^\epsilon_0$ denote the pullback connections on $\widetilde M$, and let
      \begin{equation*}
        N_\epsilon^\Phi= \pi_*\Phi \big ( \Theta (D^\epsilon_t)\big ).
        \end{equation*}
        Then by standard arguments (see, e.g., \cite[Chapter III.3]{wells}),
         $N_\epsilon^\Phi$ is a form of degree $2n-1$ such that
        \begin{equation}\label{alphabet}
          d N_\epsilon^\Phi= \Phi \big ( \Theta (\widehat D^\epsilon_0)\big )-
          \Phi \big ( \Theta (D^\epsilon_0)\big ),
        \end{equation}
cf.\ the proof of Proposition ~\ref{prop:changeConnectionsMetrics}.

Recall from the proof of Theorem ~\ref{train} that $\Theta (\widehat
D^\epsilon_0)$ is of the form \eqref{bluewear}; by the same arguments
$\Theta (D^\epsilon_0)$ is as well of the form \eqref{bluewear}. It
follows as in the proof of Proposition
~\ref{prop:changeConnectionsMetrics} that the limit of
$N_\epsilon^\Phi$ exists as a pseudomeromorphic current $N^\Phi$.
Moreover, since $\widehat D_0^\epsilon$ and $D^\epsilon_0$ are
$(1,0)$-connections, see Remark ~\ref{axel} and \ref{newyear}, as in the proof of Proposition ~\ref{prop:changeConnectionsMetrics}, it follows that $N^\Phi$ is of degree $(n, n-1)$.
Taking limits in \eqref{alphabet} we get \eqref{betapet}.

Since $\widehat D_0^\epsilon=D_0^\epsilon=D_0$ in $M\setminus
\Sigma_\epsilon$, see Section ~\ref{evening} and Remark \ref{newyear}, and $D_0$ is an
$X$-connection there by Lemma ~\ref{roadtrip}, it follows from
Lemma~8.18 in \cite{baum-bott} that $N_\epsilon^\Phi$ has support in
$\Sigma_\epsilon$. As in previous proofs we may assume that
$\{s=0\}=\{X=0\}$; in fact, we can choose $s=X$. Hence $N^\Phi$ has support on $\{X=0\}$.
    \end{proof}

    \begin{proof}[Proof of Lemma ~\ref{step2}]
      Throughout this proof we write $\sigma=\sigma_1$.

      Since $D^\epsilon_0$ is a $(1,0)$-connection, see Remark \ref{newyear}, and $\Phi$ is of degree $n$, it follows that
      $$\Phi\big(\Theta({D}_0^\epsilon)\big ) = \Phi\big (\Theta({D}_0^\epsilon)_{(1,1)}\big ),
      $$
      where $(\, \cdot \,)_{(1,1)}$ denotes the component of bidegree $(1,1)$.
By \eqref{eq:Dvarphi1} and \eqref{eq:Dvarphi1-epsilon}
\begin{equation*}
D_0^\epsilon = D^{TM} + \chi_{\epsilon} \Dc \varphi_1 \sigma (dz \cdot \partial/\partial z) = D^{TM} + \chi_{\epsilon} \Dc \varphi_1 \widetilde \sigma,
\end{equation*}
cf.\ \eqref{eq:tildeSigma}.
Since $\chi_{\epsilon} \Dc \varphi_1 \widetilde \sigma$ has bidegree
$(1,0)$, see Section \ref{pomona},
it follows that
    \begin{equation}\label{mountain}
        \Theta(D_0^\epsilon)_{(1,1)} = \Theta(D^{TM})_{(1,1)} +\dbar(\chi_\epsilon \Dc \varphi_1 \widetilde{\sigma}).
      \end{equation}
   Thus, by \eqref{eq:PhiAdd}, using that the forms in \eqref{mountain} are $\End (TM)$-valued $2$-forms,
    \begin{equation*}
       \Phi(\Theta(D_0^\epsilon)) = \sum_{k=0}^n \Phi_k\big (\dbar(\chi_\epsilon \Dc\varphi_1\widetilde{\sigma}),\Theta(D^{TM})_{(1,1)}\big ).
    \end{equation*}
    Since $\widetilde{\sigma}$ is a scalar-valued $1$-form,  $\widetilde{\sigma}\wedge \widetilde{\sigma}= 0$. Moreover $\dbar\chi_\epsilon\wedge\dbar\chi_\epsilon =0$. Using this we get that
    \begin{multline}\label{earlymorning}
      \Phi_k \big (\dbar(\chi_\epsilon \Dc\varphi_1\widetilde{\sigma}),\Theta(D^{TM})_{(1,1)}\big )= \\
      \chi_\epsilon^k ~\Phi_k\big (\dbar(\Dc\varphi_1 \widetilde{\sigma}) ,\Theta(D^{TM})_{(1,1)}\big ) +
      k ~\dbar\chi_\epsilon \wedge \widetilde{\sigma}\wedge (\chi_\epsilon \dbar \widetilde{\sigma})^{k-1}
\Phi_k\big (\Dc\varphi_1,\Theta(D^{TM})_{(1,1)}\big ).
\end{multline}

Let us consider the contribution to $\Phi(\Theta(D_0^\epsilon))$ from the first term in the right hand side of \eqref{earlymorning}. In view of Remarks ~\ref{elsa} and ~\ref{bronze},
\begin{equation}\label{broken}
  \lim_{\epsilon \to 0} \sum_k
   \chi_\epsilon^k ~\Phi_k\big (\dbar(\Dc\varphi_1 \widetilde{\sigma}),\Theta(D^{TM})_{(1,1)}\big )
   =\\
  \lim_{\epsilon \to 0} \chi_\epsilon ~\sum_k
\Phi_k\big (\dbar(\Dc\varphi_1 \widetilde{\sigma}),\Theta(D^{TM})_{(1,1)}\big ).
\end{equation}
Next, by \eqref{eq:PhiAdd} and \eqref{eq:Dvarphi1}, cf. \eqref{eq:tildeSigma},
 \begin{equation*}
   \sum_k
\Phi_k\big (\dbar(\Dc\varphi_1 \widetilde{\sigma}),\Theta(D^{TM})_{(1,1)}\big ) = \Phi\big (\Theta(D_0)\big).
\end{equation*}
Since $D_0$ is an $X$-connection over $M \setminus \{X=0\}$ by Lemma ~\ref{roadtrip}, $\Phi(\Theta(D_0)) = 0$ there by Lemma ~\ref{lemma:X-vanishing}. Thus, since $\chi_\epsilon$ vanishes in a neighborhood of $\{X=0\}$ we get that \eqref{broken} vanishes identically on $M$.

Next, let us consider the second term in the right hand side of \eqref{earlymorning}.
In view of Remarks ~\ref{elsa} and ~\ref{bronze} and \eqref{eq:tildeRX},
\begin{equation*}
  \lim_{\epsilon \to 0} k\dbar\chi_\epsilon \wedge \widetilde{\sigma}\wedge (\chi_\epsilon \dbar \widetilde{\sigma})^{k-1}
  =
   \lim_{\epsilon \to 0}  \dbar\chi^k_\epsilon \wedge \widetilde{\sigma}\wedge (\dbar \widetilde{\sigma})^{k-1}= \widetilde R_k^X.
  \end{equation*}
Hence, since $\Phi_k(\Dc\varphi_1,\Theta(D^{TM})_{(1,1)})$ is a smooth
form,
\begin{equation}\label{vormittag}
 \lim_{\epsilon \to 0} k \dbar\chi_\epsilon \wedge \widetilde{\sigma}\wedge (\chi_\epsilon \dbar \widetilde{\sigma})^{k-1}
 \Phi_k\big (\Dc\varphi_1,\Theta(D^{TM})_{(1,1)}\big ) =
 \widetilde R_k^X\wedge \Phi_k\big (\Dc\varphi_1,\Theta(D^{TM})_{(1,1)}\big ).
\end{equation}
Since $\widetilde{R}^X_k$ is a pseudomeromorphic current of bidegree $(k,k)$, it vanishes by the dimension principle when $k < \codim \{X=0\}$, see Section ~\ref{hemma}.  Also, since $D^{TM}$ is a $(1,0)$-connection,
for degree reasons, $\Theta(D^{TM})_{(1,1)}$ may be replaced by $\Theta(D^{TM})$ in \eqref{vormittag}.
We conclude that $\Phi(\Theta(D_0^\epsilon))$ is of the form \eqref{mantel}.
\end{proof}

\smallskip

From Theorem ~\ref{thm:BBBM} we obtain the following simple expression
of the Baum-Bott currents of an isolated singularity of a rank one
foliation.

\begin{cor}\label{nordstan}
  Let $M$ be a complex manifold of dimension $n$, let $\F$ be a rank one foliation on $M$, and let $\Phi\in \C[z_1,\ldots, z_n]$ be a homogeneous symmetric polynomial of degree $n$. Consider the resolution
  \eqref{eq:resolution-vector-field} and assume that $TM$ and $L$ are equipped with Hermitian metrics and $(1,0)$-connections $D^{TM}$ and $D_1$, respectively, and assume that $D^{TM}$ is torsion free. Assume that $p$ is an isolated singularity of $\F$ and let $R^\Phi_{\{p\}}$ be the associated Baum-Bott current. Let $z = (z_1,\ldots,z_n)$ be a local coordinate system centered at $p$ so that $\F$ is generated by the vector field
\begin{equation}\label{breakfast}
  X = \sum a_i(z) \frac{\partial}{\partial z_i}
  \end{equation} near $p$.  Then
    \begin{equation}\label{dinner}
      R_{\{p\}}^\Phi = \frac{1}{(2\pi i)^n}~
      \dbar\frac{1}{a_n}\wedge \dots \wedge\dbar\frac{1}{a_1}\wedge \Phi\left(\left(\frac{\partial a_i}{\partial z_j}\right)_{ij}\right) dz_1 \wedge \cdots \wedge dz_n.
    \end{equation}
  \end{cor}

\begin{remark} \label{rmk:grothendieck-residue}
  Note in view of Example ~\ref{supper} that the action of $R_{\{p\}}^\Phi$ on the function $1$ equals
  \[
    \Res_p\left[ \Phi\Big( \big(\frac{\partial a_i}{\partial z_j}\big)_{ij}\Big) \frac{dz_1 \wedge \ldots \wedge dz_n}{a_1\cdots a_n} \right]
    \]
  In particular, we recover the classical expression of Baum-Bott residues
  in terms of the Grothendieck residue given in \eqref{eq:grothendieck-residue}, see also \cite[Theorem 1 and Proposition 8.67]{baum-bott}.
\end{remark}

\begin{proof}[Proof of Corollary ~\ref{nordstan}]

Since $R_{\{p\}}^\Phi$ only depends on \eqref{eq:resolution-vector-field-2} in a neighborhood of $p$ we may replace $M$ by a neighborhood of $p$ where $L$ is trivial and $\F$ is generated by a vector field $X$ of the form \eqref{breakfast}, so that we are in the situation of Theorem ~\ref{thm:BBBM}. Note that $\Phi_n (\Dc \varphi_1, \Theta (D^{TM}))= \Phi (\Dc \varphi_1)$, cf.\ \eqref{liseberg}. Thus, it follows from Theorem ~\ref{thm:BBBM} that
  \begin{equation}\label{syrup}
    R^\Phi_{\{p\}} =  \left (\frac{i}{2\pi}\right )^n \widetilde{R}_n^X \wedge \Phi(\Dc \varphi_1).
    \end{equation}
    Indeed, the pseudomeromorphic current $N^\Phi$ vanishes by the dimension principle, see Section ~\ref{hemma}, since it has bidegree $(n, n-1)$ and support on the point $p$.

    Let us make the right hand side in \eqref{syrup} more explicit.
By Corollary ~\ref{famine}, $R_{\{p\}}^\Phi$ is independent of the choice of
Hermitian metrics and $(1,0)$-connections. We may therefore assume that $D^{TM}$ and $D_1$ are trivial.
    If $D$ is the connection on \eqref{eq:resolution-vector-field-2} induced by $D^{TM}$ and $D_1$, then $D_{\End} \varphi_1=d\varphi_1$.
    Since $L$ is trivial $\Dc$ can be regarded as a section of $\End(TM)$.
    By \eqref{jewel},
    $
    \Dc \varphi_1(u) = -i(u) d\varphi_1.
    $
    Recall that $\varphi_1$ is just multiplication by $X$.
    A computation yields that $i(u) d\varphi_1$ is multiplication by the Jacobian matrix $ \big( \frac{\partial a_i}{\partial z_j} \big)_{ij}$, cf.\ \eqref{jewel}.
   Thus
   \begin{equation}\label{varberg}
     \Phi (\Dc \varphi_1) =
(-1)^n  \Phi\left(\left(\frac{\partial a_i}{\partial z_j}\right)_{ij}\right).
\end{equation}
By plugging \eqref{halmstad} and \eqref{varberg} into \eqref{syrup}, we get
\eqref{dinner}.
    \end{proof}

\end{document}